\documentclass[11pt]{amsart} 
\RequirePackage[utf8]{inputenc}

\usepackage{amsmath,amsthm,amsfonts,amssymb,amscd,amsbsy,dsfont,hyperref}
\usepackage[all]{xy}

\usepackage{palatino}

\newcommand{\dd}{\mathrm{d}}

\newcommand{\Ric}{\operatorname{Ric}}
\newcommand{\cut}{\operatorname{Cut}}
\newcommand{\g}{\mathrm g}
\newcommand{\cvc}{\operatorname{cvc}}
\newcommand{\spn}{\operatorname{span}}
\newcommand{\inj}{\operatorname{inj}}

\newcommand{\cl}{\operatorname{closure}}
\newcommand{\Id}{\operatorname{Id}}
\newcommand{\fc}{\operatorname{FConj}}
\newcommand{\fix}{\operatorname{Fix}}
\newcommand{\conj}{\operatorname{conj}}
\newcommand{\diam}{\operatorname{diam}}

\newcommand{\R}{\mathbb R}

\newcommand{\Sp}{\operatorname{Sp}}

\renewcommand{\geq}{\geqslant}
\renewcommand{\leq}{\leqslant}

\usepackage{geometry}
\allowdisplaybreaks

\hypersetup{
    pdftoolbar=true,
    pdfmenubar=true,
   pdffitwindow=false,
    pdfstartview={FitH},
    pdftitle={Positively curved manifolds with large spherical rank}, 
    pdfauthor={Benjamin Schmidt, Krishnan Shankar, and Ralf Spatzier}, 
    pdfsubject={}, 
    pdfkeywords={},
    pdfnewwindow=true,
    colorlinks=true, 
    linkcolor=blue,
    citecolor=blue,
    urlcolor=black,
}

\newtheorem{theorem}{Theorem}[]
\newtheorem{lemma}[theorem]{Lemma}
\newtheorem{proposition}[theorem]{Proposition}
\newtheorem{corollary}[theorem]{Corollary}

\newtheorem{mainthm}{\sc Theorem}
\newtheorem{maincor}[mainthm]{\sc Corollary}

\theoremstyle{definition}
\newtheorem{definition}[theorem]{Definition}

\theoremstyle{remark}
\newtheorem{remark}[theorem]{Remark}
\newtheorem{convention}[theorem]{Convention}
\newtheorem{example}[theorem]{Example}

\title{Positively Curved Manifolds with Large Spherical Rank}
\author[B. Schmidt]{Benjamin Schmidt}
\author[K. Shankar]{Krishnan Shankar}
\author[R. Spatzier]{Ralf Spatzier}

\address{\begin{tabular}{lll}
Michigan State University & University of Oklahoma & Univeristy of Michigan  \\
Dept. of Mathematics & Dept. of Mathematics & Dept. of Mathematics \\
619 Red Cedar Road & 601 Elm Avenue & 530 Church Street  \\ 
East Lansing, MI, 48824 & Norman, OK, 73019 & Ann Arbor, MI, 48109\\
{\tt schmidt@math.msu.edu} & {\tt shankar@math.ou.edu} & {\tt spatzier@umich.edu}
\end{tabular}
}

\allowdisplaybreaks
\numberwithin{equation}{section}
\numberwithin{theorem}{section}

\thanks{The first named author is partially supported by the NSF grant DMS--1207655.  The second named author is partially supported by the NSF grant DMS--1104352.  The third named author is partially supported by the NSF grant DMS--1307164.}

\date{\today}

\begin{document}

\newcommand{\spacing}[1]{\renewcommand{\baselinestretch}{#1}\large\normalsize}
\spacing{1.2}

\begin{abstract}
Rigidity results are obtained for Riemannian $d$-manifolds with $\sec \geq 1$ and spherical rank at least $d-2>0$. Conjecturally, all such manifolds are locally isometric to a round sphere or complex projective space with the (symmetric) Fubini--Study metric.  This conjecture is verified in all odd dimensions, for metrics on $d$-spheres when $d \neq 6$, for Riemannian manifolds satisfying the Raki\'c duality principle, and for K\"ahlerian manifolds.  
\end{abstract}

\maketitle

\section{Introduction}
A complete Riemannian $d$-manifold $M$ has \textit{extremal curvature} $\epsilon\in  \{-1,0, 1\}$ if its sectional curvatures satisfy $\sec \leq \epsilon$ or $\sec \geq \epsilon$.  For $M$ with extremal curvature $\epsilon$, the rank of a complete geodesic $\gamma: \mathbb{R} \rightarrow M$ is defined as the maximal number of linearly independent, orthogonal, and parallel vector fields $V(t)$ along $\gamma(t)$ satisfying $\sec(\dot{\gamma},V)(t)\equiv \epsilon$.  The manifold $M$ has (hyperbolic, Euclidean or spherical according as $\epsilon$ is $-1, 0$ or 1) rank at least $k$ if all its complete geodesics have rank at least $k$.


Riemannian manifolds with $\sec \leq \epsilon$ and admitting positive rank are known to be rigid.
Finite volume Riemannian manifolds with bounded nonpositive sectional curvatures and positive Euclidean rank are locally reducible or locally isometric to symmetric spaces of nonpositive curvature \cite{ba, busp}.  Generalizations include \cite{ebhe} and \cite{wa}.  Closed Riemannian manifolds with $\sec \leq -1$ and positive hyperbolic rank are locally isometric to negatively curved symmetric spaces \cite{ha};  this fails in infinite volume \cite{co}.  Finally, closed Riemannian manifolds with $\sec \leq 1$ and positive spherical rank are locally isometric to positively curved, compact, rank one symmetric spaces \cite{shspwi}.

Rank rigidity results are less definitive in the $\sec \geq \epsilon$ curvature settings.  Hyperbolic rank rigidity results for manifolds with $-1\leq \sec \leq 0$ first appeared in \cite{con}.  Finite volume $3$-manifolds with $\sec \geq -1$ and positive hyperbolic rank are real hyperbolic \cite{scwo}.  Complete Riemannian $3$-manifolds with $\sec \geq 0$ and positive Euclidean rank have reducible universal coverings as a special case of \cite{besc}, while the higher dimensional  $\sec \geq 0$ examples in \cite{spst}, \cite{heintze} illustrate that rank rigidity does not hold in complete generality. 

Our present focus is the curvature setting $\sec \geq 1$.  Conjecturally, manifolds with $\sec \geq 1$ and positive spherical rank are locally isometric to positively curved symmetric spaces.  Note that the simply connected, compact, rank one symmetric spaces, normalized to have minimum sectional curvature 1, have spherical rank: $n-1 = \dim(S^n) - 1$ for the spheres; $2n-2 = \dim(\mathbb{CP}^n) - 2$ for complex projective space; $4n-4 = \dim(\mathbb{HP}^n) - 4$ for quaternionic projective space; $8 = \dim(\mathbb{OP}^2) - 8$ for the Cayley projective plane.
Our main theorems concern $d$-manifolds with spherical rank at least $d-2$, spaces that are conjecturally locally isometric to spheres or complex projective spaces. 

\begin{mainthm}\label{thm:A}
An odd dimensional Riemannian $d$-manifold with $d\geq 3$, $\sec \geq 1$, and spherical rank at least $d-2$ has constant sectional curvatures $\sec \equiv 1$.
\end{mainthm}

\begin{mainthm}\label{thm:B}
Let $M$ be an even dimensional Riemannian $d$-manifold with $d\geq 4$, $\sec \geq 1$, and spherical rank at least $d-2$.  If $M$ does not have constant sectional curvatures i.e., $\sec \not\equiv 1$, then $M$ satisfies:
\begin{enumerate}
\item Every vector $v \in SM$ is contained in a $2$-plane section $\sigma$ with $\sec(\sigma)>1$.
\item The geodesic flow $\phi_t:SM \rightarrow SM$ is periodic with $2\pi$ a period.
\item There exists an almost complex structure $J:TM\rightarrow TM$ if $M$ is simply connected. 
\item If $M$ is simply connected and if $\sec < 9$, then every geodesic in $M$ is simple, closed, and of length $\pi$.  Moreover, $M$ is homotopy equivalent to $\mathbb{C}\mathbb{P}^{d/2}.$ \end{enumerate}
\end{mainthm}

A Riemannian manifold satisfies the \textit{Raki\'c duality principle} if for each $p \in M$, orthonormal vectors $v,w \in S_pM$, and $c \in \R$,  $v$ lies in the $c$-eigenspace of the Jacobi operator $\mathcal{J}_w$ if and only if $w$ lies in the $c$-eigenspace of the Jacobi operator $\mathcal{J}_v$.  This property arises naturally in the study of Osserman manifolds \cite{nira, ra}. See Section \ref{prelim} for details.

\begin{mainthm}\label{thm:C}
Let $M$ be a Riemannian $d$-manifold with $\sec \geq 1$ and spherical rank at least $d-2$.  If $M$ satisfies the Raki\'c duality principle, then $M$ is locally symmetric. 
\end{mainthm}

\begin{mainthm}\label{thm:D}
A K\"ahlerian manifold with $\sec \geq 1$, real dimension $d \geq 4$, and spherical rank at least $d-2$ is isometric to a symmetric $\mathbb{C}\mathbb{P}^{d/2}$ with holomorphic curvatures equal to $4$.
\end{mainthm}

Theorem \ref{thm:A} implies: 

\begin{maincor}\label{cor:E}
A Riemannian $3$-manifold with $\sec \geq 1$ and positive spherical rank has constant sectional curvatures.
\end{maincor}

Only the two- and six-dimensional spheres admit almost complex structures \cite{bose}. Hence, item (3) in Theorem \ref{thm:B} implies:

\begin{maincor}\label{cor:F}
A Riemannian sphere $S^d$ with $d \neq 2, 6$, $\sec \geq 1$, and with spherical rank at least $d-2$ has constant sectional curvatures. 
\end{maincor}


It is instructive to compare the $\sec \geq 1$ case considered here with that of the $\sec \leq 1$ case of rank-rigidity resolved in \cite{shspwi}.  In both cases, each unit-speed geodesic $\gamma: \mathbb{R} \rightarrow M$ admits a Jacobi field $J(t)=\sin(t)V(t)$ where $V(t)$ is a normal parallel field along $\gamma$ contributing to its rank .  Hence, for each $p \in M$, the tangent sphere of radius $\pi$ is contained in the singular set for $\exp_p:T_pM \rightarrow M$.  In a symmetric space with $\frac{1}{4}\leq \sec \leq 1$, the first conjugate point along a unit-speed geodesic occurs at time $\pi$, the soonest time allowed by the curvature assumption $\sec \leq 1$.   Consequently, the rank assumption is an assumption about the locus of \textit{first singularities} of exponential maps when $\sec \leq 1$.  In symmetric spaces with $1\leq \sec \leq 4$, the first  and second conjugate points along a unit-speed geodesic occur at times $\pi/2$ and $\pi$, respectively.  Therefore, when rank-rigidity holds in the $\sec \geq 1$ setting, the rank assumption is an assumption about the locus of \textit{second singularities} of exponential maps.  Concerning \textit{first singularities}, a simply-connected Riemannian manifold with $\sec \geq 1$ in which the first conjugate point along each unit-speed geodesic occurs at time $\pi/2$ is globally symmetric \cite{sc}.


An alternative definition for the spherical rank of a geodesic $\gamma$ in a Riemannian manifold with $\sec \geq 1$ is the dimension of the space of \textit{normal Jacobi fields} along $\gamma$ that make curvature one with $\gamma$.  This alternative notion of rank is a priori less restrictive since parallel fields $V(t)$ give rise to Jacobi fields $J(t)$ as described above.  The Berger spheres, suitably rescaled, have positive rank when defined in terms of Jacobi fields \cite{shspwi} but not when defined in terms of parallel fields by Corollary \ref{cor:E}.  Moreover, there is an infinite dimensional family of Riemannian metrics on $S^3$ with $\sec\geq 1$ and positive rank when defined in terms of Jacobi fields \cite{scwo2}.  In particular, there exists examples that are not locally homogeneous.  Each such metric admits a unit length Killing field $X$ with the property that a $2$-plane section $\sigma \subset TM$ with $X\in \sigma$ has $\sec(\sigma)=1$; the restriction of $X$ to a geodesic is a Jacobi field whose normal component contributes to the rank.  There are no known examples with discrete isometry group.


To describe our methods and the organization of the paper, let $\mathcal{I}=\{p \in M \, \vert\, \sec_p \equiv 1\}$ and $\mathcal{O}=M \setminus \mathcal{I}$ denote the subsets of isotropic and nonisotropic  points in $M$, respectively.  The goal is to prove that $M$ is locally isometric to complex projective space when $\mathcal{O}\neq \emptyset.$

We start with a pointwise analysis of curvature one planes.  Given a vector $v \in S_pM$, let $E_v$ denote the span of all vectors $w$ orthogonal to $v$ with $\sec(v,w)=1$ and let $D_v$ denote the subspace of $E_v$ spanned by vectors contributing to the rank of the geodesic $\gamma_v(t)$.  The assignments $v \mapsto E_v$ and $v \mapsto D_v$ define two (possibly singular) distributions on each unit tangent sphere $S_pM$, called the \textit{eigenspace} and \textit{spherical} distributions, respectively (see \ref{use} and \ref{spheredist}).  The spherical rank assumption ensures that $d-2=\dim(S_pM)-1\leq \dim(D_v)$ for each $v \in S_pM$ so that both distributions are of codimension at most one on $S_pM$.

The arrangement of curvature one planes at nonisotropic points $p$ encodes what ought to be a complex structure, a source of rigidity.  More precisely, the eigenspace distribution on $S_pM$ is \textit{totally geodesic} (see Lemma \ref{totgeo}) and of codimension at most one. Subsection \ref{sec:totgeo} builds on earlier work of Hangan and Lutz \cite{halu} where they exploited the fundamental theorem of projective geometry to prove that  codimension one totally geodesic distributions on odd dimensional spheres are \textit{algebraic}: there is a nonsingular projective class $[A]$ of skew-symmetric linear maps of $\mathbb{R}^{n+1}$ with the property that the distribution is orthogonal to the Killing (line) field on $S^n$ generated by $[A]$.  In particular, such distributions are projectively equivalent to the standard contact hyperplane distribution. Note that when $M$ is complex projective space, with complex structure $J:TM \rightarrow TM$,  the codimension one eigenspace distribution on $S_pM$ is orthogonal to the Killing (line) field on $S_pM$ generated by $[J_p]$.


As the spherical distribution $D$ is invariant under parallel transport along geodesics ($D_{\dot{\gamma}_{v}(t)}=P_t(D_v)$), its study leads to more global considerations in Section \ref{sec:3.1}.  The sphere of radius $\pi$ in $T_pM$ is also equipped with a \textit{kernel distribution},  $v \mapsto K_v:=\ker(\dd(\exp_p)_v)$ (see \ref{kern}).  As each $w \in D_v$ is an initial condition for an initially vanishing spherical Jacobi field along $\gamma_v(t)$, parallel translation in $T_pM$ identifies the spherical subspace $D_v$ with a subspace of $K_{\pi v}$ for each $v\in S_pM$ (see Lemma \ref{spherjab}). When $p \in \mathcal{O}$, the eigenvalue and spherical distributions on $S_pM$ coincide (see Lemma \ref{same}).  As a consequence, the kernel distribution contains a totally geodesic subdistribution of codimension at most one on $S(0,\pi)$.  It follows that $\exp_p$ is constant on $S(0,\pi)$ (see Corollary \ref{point}) and that geodesics passing through nonisotropic points $p \in \mathcal{O}$ are all closed (see Lemma \ref{closed1}). Moreover, when $p \in \mathcal{O}$, each vector $v \in S_pM$ has rank exactly $d-2$ (see Lemma \ref{cartan2}), or putting things together, the eigenspace distribution is a \textit{nonsingular} codimension one distribution on $S_pM$.  As even dimensional spheres do not admit such distributions, $M$ must have even dimension, proving Theorem \ref{thm:A}.  More generally, this circle of ideas and a connectivity argument culminate in a proof that every vector in $M$ has rank $d-2$ when the nonisotropic set $\mathcal{O}\neq \emptyset$ (see Proposition \ref{constantrank}).

The remainder of the paper is largely based on curvature calculations in radial coordinates with respected to frames \textit{adapted} to the spherical distributions that are introduced in Section \ref{sec:adapt}.  An argument based on these calculations and the aforementioned fact that the spherical distributions are contact distributions, establishes that if the nonisotropic set $\mathcal{O}\neq \emptyset$ , then $M=\mathcal{O}$ (see Proposition \ref{noisotropic}).  The proof of Theorem \ref{thm:B} follows easily and appears in Section \ref{sec:B}.  The proof of Theorem \ref{thm:C} appears in Section \ref{sec:C}.  There, the Raki\'c duality hypothesis is applied to prove that the family of skew-symmetric endomorphisms $A_p:T_pM \rightarrow T_pM$, $p \in M$, arising from the family of eigenspace distributions on the unit tangent spheres $S_pM$, define an almost complex structure on $M$ (see Lemma \ref{Asquared} ).  This fact, combined with additional curvature calculations in adapted framings, allows us to deduce that $M$ is Einstein, from which the theorem easily follows (see the proof of Proposition \ref{rakic}).

Finally, Sections \ref{sec:D1} and \ref{sec:D2} contain the proofs of Theorem \ref{thm:D} in real dimension at least six and in real dimension four, respectively.  The methods are largely classical, relying on pointwise curvature calculations based on the K\"ahler symmetries of the curvature tensor and on expressions for the curvature tensor when evaluated on an orthonormal $4$-frame due to Berger \cite{ber2, ka}.  Essentially, these calculations yield formulas that relate the eigenvalues of the endomorphisms $A_p:T_pM \rightarrow T_pM$ to the curvatures of eigenplanes in invariant four dimensional subspaces of $T_pM$.  When the real dimension is at least six, there are enough invariant four dimensional subspaces to deduce that $M$ has constant holomorphic curvatures, concluding the proof in that case.  The argument in real dimension four proceeds differently by proving that $M$ satisfies the Raki\'c duality principle.  When this fails, the decomposition of $TM$ into eigenplanes of $A:TM \rightarrow TM$ is shown to arise from a metric splitting of $M$, contradicting the curvature hypothesis $\sec \geq 1$.





\section{Notation and Preliminaries}\label{prelim}

This section contains preliminary results, mostly well-known, that are used in subsequent sections.  Throughout $(M,\g)$  denotes a smooth, connected, and complete $d$-dimensional Riemannian manifold,  $\mathcal{X}(M)$ the $\R$-module of smooth vector fields on $M$, and $\nabla$ the Levi-Civita connection. Let $X,Y,Z,W \in \mathcal{X}(M)$ be vector fields. Christoffel symbols for the connection $\nabla$ are determined by Koszul's formula
\begin{eqnarray}\label{koszul}
\g(\nabla_{X} Y,Z)&=&\tfrac{1}{2}\{X\g(Y,Z)+Y\g(Z,X)-Z\g(X,Y)\}\\
&+&\tfrac{1}{2}\{\g([X,Y],Z)-\g([Y,Z],X)+\g([Z,X],Y)\}.\nonumber
\end{eqnarray}  
The curvature tensor $R:\mathcal{X}(M)^3 \rightarrow \mathcal{X}(M)$ is defined by $R(X,Y)Z=[\nabla_{X}, \nabla_{Y}]Z-\nabla_{[X,Y]}Z$ and has the following symmetries
 \begin{equation}\label{symmetry}
R(X,Y,Z,W)=-R(Y,X,Z,W)=R(Z,W,X,Y)
\end{equation} where $R(X,Y,Z,W)=\g(R(X,Y)Z,W).$  The sectional curvature of a $2$-plane section $\sigma$ spanned by vectors $v$ and $w$ is defined by $\sec(\sigma)=\sec(v,w)=\frac{R(v,w,w,v)}{\|v \wedge w\|^2}$.  An \textit{almost Hermitian structure} on $M$ is an almost complex structure $J:TM \rightarrow TM$ \textit{compatible} with the metric: $\g(X,Y)=\g(JX,JY)$ for all $X, Y \in \mathcal{X}(M)$.  A \textit{Hermitian} structure on $M$ consists of an integrable almost Hermitian structure.  The K\"ahler form is the $2$-form $\omega$ defined by $\omega(X,Y)=\g(JX,Y)$. A \textit{K\"ahler structure} on $M$ consists of a Hermitian structure with closed K\"ahler form, $d\omega=0$, or equivalently, a parallel complex structure, $\nabla J=0$.  If $M$ is K\"ahlerian, then $\nabla_Y JX=J \nabla_Y X$ for all $X,Y\in \mathcal{X}(M)$, yielding the additional curvature identities

 \begin{eqnarray}
&R(X,Y,Z,W)=R(JX, JY, Z, W)=\nonumber \\ 
&R(X, Y, JZ, JW)=R(JX,JY,JZ,JW).\label{ghost}
\end{eqnarray}

These curvature identities are the key properties of a K\"ahlerian manifold used in the proof of Theorem \ref{thm:D}.

\subsection{Jacobi operators and eigenspace distributions.}
Let $SM$ denote the unit sphere bundle of $M$; its fiber above a point $p \in M$ is the unit sphere $S_pM$ in $T_pM$.  For $v \in S_pM$ define the Jacobi operator $\mathcal J_v:v^{\perp} \rightarrow v^{\perp}$ by $\mathcal J_v(w)=R(w,v)v$.  The symmetries (\ref{symmetry}) imply that $\mathcal J_v$ is a well-defined self-adjoint linear map of $v^{\perp}$.  Its eigenvalues encode the sectional curvatures of $2$-plane sections containing the vector $v$.

\begin{lemma}\label{crit}
Let $v,w \in S_pM$ be orthonormal vectors and assume that $\sec_p \geq \epsilon$ for some $\epsilon \in \R$. The following are equivalent:

\begin{enumerate}
 \item $\sec(v,w)=\epsilon$
\item $w$ is an eigenvector of $\mathcal J_v$ with eigenvalue $\epsilon$.  
\item $R(w,v)v=\epsilon w$ 
\end{enumerate}
\end{lemma}

\begin{proof}
Only $(1) \implies (2)$ is nontrivial. If $\{e_i\}_{i=1}^{n-1}$ is an orthonormal eigenbasis of $\mathcal J_v$ with corresponding eigenvalues $\lambda_i$, then $\lambda_i \geq \epsilon$ for each index $i$.  Express $w=\sum_{i=1}^{n-1} \alpha_i e_i$ with $\sum_{i=1}^{n-1} \alpha_i^2=1$.  Then $\epsilon=\g(R(w,v)v,w)=\g(\mathcal J_v(w),w)=\sum_{i=1}^{n-1} \alpha_i^2 \lambda_i.$  Conclude that $\alpha_i=0$ for indices $i$ with $\lambda_i>\epsilon$.  Therefore $w$ is an eigenvector of $\mathcal J_v$ with eigenvalue $\epsilon$.
\end{proof}

\begin{remark}
An analogous proof works when $\sec_p \leq \epsilon$.  
\end{remark}

\begin{lemma}\label{maximal}
Let $v, w \in S_pM$ be orthonormal vectors.  If $w^{\perp} \cap v^{\perp}$ consists of eigenvectors of $\mathcal J_v$, then $w$ is an eigenvector of $\mathcal J_v$. Consequently, $R(v,w,w',v)=\g(\mathcal{J}_v(w),w')=0$ for any $w' \in w^{\perp} \cap v^{\perp}$. 
\end{lemma}

\begin{proof}
The orthogonal complement to an invariant subspace of a self-adjoint operator is an invariant subspace.
\end{proof}

\subsubsection{Specialization to manifolds with $\cvc(\epsilon)$.}
\begin{definition}
A Riemannian manifold has \textit{constant vector curvature} $\epsilon$, denoted by $\cvc(\epsilon)$, provided that $\epsilon$ is an extremal sectional curvature for $M$ ($\sec \leq \epsilon$ or $\sec \geq \epsilon$) and $\epsilon$ is an eigenvalue of $\mathcal{J}_v$ for each $v \in SM$ \cite{scwo}.
\end{definition}
For each $v \in SM$, let $E_v\subset v^{\perp}$ denote the (nontrivial) $\epsilon$-eigenspace of $\mathcal{J}_v$.

\begin{convention}
For each $v\in S_pM$, parallel translation in $T_pM$ defines an isomorphism between the subspace $v^{\perp}$ of $T_pM$ and the subspace $T_v (S_pM)$ of $T_v (T_pM)$.  This isomorphism is used without mention when contextually unambiguous.
\end{convention}

\begin{convention}
Given a manifold $M$, an assignment $M \ni p \mapsto D_p\subset T_pM$ of tangent subspaces is a \textit{distribution}.  The rank of the  subspaces may vary with $p \in M$ and the assignment is not assumed to have any regularity.  The \textit{codimension} of a distribution $D$ is defined as the greatest codimension of its subspaces.  When a distribution $D$ is known to have constant rank, it is called a \textit{nonsingular distribution}.  
\end{convention}

\begin{definition}\label{use}
The \textit{$\epsilon$-eigenspace distribution} on $S_pM$, denoted by $E$, is the distribution of tangent subspaces $$S_pM \ni v \mapsto E_v \subset T_v(S_pM).$$  Its \textit{regular set}, denoted by $\mathcal{E}_p$, is the open subset of $S_pM$ consisting of unit vectors $v$ for which $\dim(E_v)$ is minimal.
\end{definition}

\begin{example}
$\epsilon$-eigenspace distributions need not have constant rank.  When $M$ is a Berger sphere suitably rescaled to have $\cvc(1)$, the curvature one $2$-planes in $S_pM$ are precisely the $2$-planes containing the Hopf vector $h\in S_pM$.  Therefore $\dim(E_h)=\dim(E_{-h})=2$, while $\dim(E_v)=1$ for any vector $v \in S_pM\setminus \{\pm h\}$. 
\end{example}

\begin{lemma}\label{chi}
For each $p \in M$ the restriction of the $\epsilon$-eigenspace distribution on $S_pM$ to $\mathcal{E}_p$ is smooth.  
\end{lemma}

\begin{proof}
The operators $\mathcal{J}_v-\epsilon \Id$ vary smoothly and have constant rank in $\mathcal{E}_p$.  Therefore the subspaces $E_v=\ker(\mathcal{J}_v-\epsilon \Id)$ vary smoothly with $v \in \mathcal{E}_p$ (see \cite[Lemma 1]{chi} for more details). 
\end{proof}

\begin{remark}\label{smoothremark}  
Let $\mathcal{E} \subset SM$ denote the collection of unit tangent vectors $v\in SM$ with $\dim(E_v)$ minimal.  The same proof as that of Lemma \ref{chi} shows that the assignment $v \mapsto E_v$ is smooth on $\mathcal{E}$.  Note that $\mathcal{E} \cap S_pM$ may not coincide with $\mathcal{E}_p$.
\end{remark}

A tangent distribution $D$ on a complete Riemannian manifold $S$ is \textit{totally geodesic} if complete geodesics of $S$ that are somewhere tangent to $D$ are everywhere tangent to $D$.

\begin{convention} Henceforth, unit tangent spheres $S_pM$ are equipped with the standard Riemannian metric, denoted by $\langle \cdot,\cdot\rangle$, induced from the Euclidean metric $\g_p(\cdot,\cdot)$ on $T_pM$.  Moreover, geodesics in $S_pM$ are typically denoted by $c$ while geodesics in $M$ are typically denoted by $\gamma$.
\end{convention}

\begin{lemma}\label{totgeo}
For each $p \in M$, the $\epsilon$-eigenspace distribution $E$ is a totally geodesic distribution on $S_pM$. 
\end{lemma}

\begin{proof}
Let $v \in S_pM$ and $w \in E_v$.  The geodesic $c(t)=\cos(t) v+\sin(t) w$ satisfies $c(0)=v$ and $\dot{c}(0)=w$.  Calculate $\mathcal J_{c(t)}(\dot{c}(t))=-\sin(t)\mathcal J_w(v)+\cos(t)\mathcal J_v(w)$.  By assumption, $\mathcal J_v(w)=\epsilon w$. By Lemma \ref{crit}, $\mathcal J_w(v)=\epsilon v$.  Therefore $\mathcal J_{c(t)}(\dot{c}(t))=\epsilon(-\sin(t)v+\cos(t)w)=\epsilon \dot{c}(t)$.  Hence $\dot{c}(t) \in E_{c(t)},$ concluding the proof.
\end{proof}

\subsection{Conjugate points and Jacobi fields.}\label{conjsub} Let $M$ denote a smooth, connected, and complete Riemannian manifold.  

\begin{convention} Henceforth, geodesics are parameterized by arclength.  Moreover, the notation $\gamma_v(t)$ is frequently used to denote a complete unit speed geodesic with initial velocity $v=\dot{\gamma}_v(0) \in S_{\gamma(0)}M$.
\end{convention}

Let $\exp_p:T_pM \rightarrow M$ denote the exponential map and $\mathbf{r}:T_pM \setminus \{0\} \rightarrow S_pM$ the radial retraction $\mathbf{r}(v)=\frac{v}{\|v\|}$.  Critical points of $\exp_p$ are  \textit{conjugate vectors}. For a conjugate vector $v \in T_pM$, let  
\begin{equation}\label{kern}
K_v=\ker(\dd(\exp_p)_v) \subset T_v (T_pM).
\end{equation}
The \textit{multiplicity} of $v$ is defined as $\dim(K_v)$.  For $t>0$, let $S(0, t)$ denote the sphere in $T_pM$ with center $0$ and radius $t$.  Gauss' Lemma asserts $K_v \subset T_v (S(0,\|v\|)).$

Let $v \in T_pM$ be a conjugate vector and $\gamma(t)=\exp_p(t\mathbf{r}(v))$. The point $q=\exp_p(v)$ is \textit{conjugate} to the point $p$ along $\gamma$ at time $t=\|v\|$.  The point $q=\exp_p(v)$ is a \textit{first conjugate point} to $p$ along $\gamma$ if $v$ is a \textit{first conjugate vector}, i.e. $tv$ is not a conjugate vector for any $t \in (0,1)$.  Denote the locus of first conjugate vectors in $T_pM$ by $\fc(p)$.  The conjugate radius at $p$, denoted $\conj(p)$, is defined by $\conj(p)=\inf_{v \in \fc(p)} \{\|v\|\}$ when $\fc(p)\neq \emptyset$ and $\conj(p)= \infty$ otherwise; when $\fc(p)\neq \emptyset$, the infimum is realized as a consequence of Lemma \ref{closed} below.  The conjugate radius of $M$, denoted $\conj(M)$, is defined by $\conj(M)=\inf_{p \in M} \{\conj(p)\}$.

Equivalently, conjugate vectors and points are described in terms of Jacobi fields along $\gamma$.  A normal Jacobi field along $\gamma(t)$ is a vector field $J(t)$, perpendicular to $\dot{\gamma}(t)$ and satisfying Jacobi's second order ode: $J''+R(J,\dot{\gamma})\dot{\gamma}=0.$  Initial conditions $J(t),J'(t) \in \dot{\gamma}(t)^{\perp}$ uniquely determine a normal Jacobi field. Let $p=\gamma(0)$, $v=\dot{\gamma}(0) \in S_pM$, and $w \in v^{\perp}$.  The geodesic variation $\alpha(s,t)=\exp_p(t(v+sw))$ of $\gamma(t)=\alpha(0,t)$ has variational field $J(t)=\frac{\partial}{\partial s} \alpha(s,t)\vert_{s=0}$, a normal Jacobi field along $\gamma$ with initial conditions $J(0)=0$ and $J'(0)=w$ given by
\begin{equation}\label{jabfield}
J(t)=\dd(\exp_p)_{tv}(tw).
\end{equation}
If $J(a)=0$, then (\ref{jabfield}) implies that $aw \in K_{av}$.  In this case $av$ is a conjugate vector and $\gamma(a)$ is a conjugate point to $p=\gamma(0)$ along $\gamma$.  All initially vanishing normal Jacobi fields along $\gamma$ arise in this fashion, furnishing the characterization: $\gamma(a)$ is conjugate to $\gamma(0)$ along $\gamma$  if and only if there exists a nonzero normal Jacobi field $J(t)$ along $\gamma$ with $J(0)=J(a)=0$. 

For $\gamma(t)$ a geodesic and $t_0>0$, let $\mathcal{V}_{\gamma}^{t_0}$ denote the vector space of piecewise differentiable normal vector fields $X(t)$ along $\gamma(t)$ with $X(0)=X(t_0)=0$.  The index form $I_{\gamma}^{t_0}:\mathcal{V}_{\gamma}^{t_0} \times \mathcal{V}_{\gamma}^{t_0} \rightarrow \R$ is the bilinear symmetric map defined by 
\begin{equation}\nonumber
I_{\gamma}^{t_0}(X,Y)=\int_{0}^{t_0}  \g(X',Y')-\g(R(X,\dot{\gamma})\dot{\gamma}, Y) \, dt.
\end{equation}

The null space of $I_{\gamma}^{t_0}$ consists of normal Jacobi fields $J(t)$ along $\gamma(t)$ with $J(0)=J(t_0)=0$.  By the Morse Index Theorem \cite[Chapter 11]{doca}, there exists $X \in \mathcal{V}_{\gamma}^{t_0}$ such that $I_{\gamma}^{t_0}(X,X)<0$ if and only if there exists $0<s<t_0$ such that $\gamma(s)$ is conjugate to $\gamma(0)$ at time $s$.  In particular, the property of being a first conjugate point along a geodesic segment is a symmetric property.

\begin{lemma}\label{closed}
$\fc(p)$ is a closed subset of $T_pM$.
\end{lemma}

\begin{proof}
Assume that $v_i \in \fc(p)$ converge to $v \in T_pM$.  Let $t_i=\|v_i\|$, $\bar{t}=\|v\|$, and $w_i=\mathbf{r}(v_i), w=\mathbf{r}(v) \in S_pM$.  As $v_i$ is a conjugate vector, there exists a normal Jacobi field $J_i(t)$ along $\gamma_{w_i}(t)$ with $J_i(0)=J_i(t_i)=0$ and $\|J_i'(0)\|=1$.  A subsequence of the Jacobi fields $J_i(t)$  converges to a nonzero Jacobi field $J(t)$ along $\gamma_w(t)$ with $J(0)=J(\bar{t})=0$.  Therefore $v$ is a conjugate vector.  If $v \notin \fc(p)$, there exists $0<s<1$ such that $sv$ is a conjugate vector. Therefore there exists $X \in \mathcal{V}_{\gamma_w}^{\bar{t}}$ with $I_{\gamma_w}^{\bar{t}}(X,X)<0$.

An orthonormal framing $\{e_1,\ldots, e_{n-1}\}$ of a neighborhood $B$ of $w$ in $S_pM$ induces parallel orthonormal framings $\{E_1(t),\ldots,E_{n-1}(t)\}$ along geodesics with initial tangent vectors in $B$, yielding isomorphisms between $\mathcal{V}_{\gamma_{b}}^{t} \cong \mathcal{V}_{\gamma_w}^{t}$ for each $b \in B$ and $t>0$.  Under these isomorphisms, $I_{\gamma_{w_i}}^{t_i}(X,X) \rightarrow I_{\gamma_w}^{\bar{t}}(X,X)$ by continuity; therefore, $I_{\gamma_{w_i}}^{t_i}(X,X)<0$ for all $i$ sufficiently large, contradicting $v_i \in \fc(p)$. 
\end{proof}

\subsection{Codimension one totally geodesic distributions on spheres.}\label{sec:totgeo}

Given a non-zero skew-symmetric linear map $A:\mathbb{R}^{d}\rightarrow \mathbb{R}^{d}$ and  $v \in S^{d-1}$, parallel translation in $\mathbb{R}^{d}$ identifies $v^{\perp}$ and $T_vS^{d-1}$.   As $A$ is skew-symmetric and non-zero, the assignment $S^{d-1} \ni v \mapsto Av \in T_vS^{d-1}$ defines a Killing field on $S^{d-1}$.  Let $E_v=\spn\{v,Av\}^{\perp}$ denote the subspace of $T_vS^{d-1}$ orthogonal to $Av$.  Then $S^{d-1} \ni v \mapsto E_v \subset T_vS^{d-1}$ defines a codimension one totally geodesic distribution on $S^{d-1}$ with singular set $\mathcal{X}:=\{x \in S^{d-1}\, \vert \, E_x=T_xS^{d-1}\}=\ker(A) \cap S^{d-1}$ as a consequence of the following well-known lemma.

\begin{lemma}\label{killing}
Let $X$ be a Killing field on a complete Riemannian manifold $(S,\g)$.  If a geodesic $c(t)$ satisfies $\g(\dot{c},X)(0)=0$, then $\g(\dot{c},X)(t)\equiv 0$.
\end{lemma}  

The skew-symmetric linear map $A$ and each nonzero real multiple $rA$ yield the same  codimension one totally geodesic distribution $E$ on $S^n$.  In \cite{halu}, Hangan and Lutz apply the fundamental theorem of projective geometry to establish the following:


\begin{theorem}[Hangan and Lutz]\label{Amap}
Let $E$ be a nonsingular codimension one totally geodesic distribution on a unit sphere $S^{d-1}\subset \mathbb{R}^{d}$.  Then $d-1$ is odd and there exists a nonsingular projective class $[A]\in PGL(\mathbb{R}^{d})$ of  skew-symmetric linear maps such that for each $x \in S^{d-1}$,  $E_x=\spn\{x,Ax\}^{\perp}$.  
\end{theorem}

The elegance of their approach lies in the fact that no a priori regularity assumption is made, while a posteriori the distribution is algebraic.  The following corollary is immediate (see \cite{halu}).

\begin{corollary}\label{contact}
A nonsingular codimension one totally geodesic distribution on an odd dimensional unit sphere is real-analytic and contact.
\end{corollary}

\begin{corollary}\label{geodesic}
Let $E$ be a nonsingular codimension one totally geodesic distribution on an odd dimensional unit sphere $S^{d-1}$.  The line field $L$ on $S^{d-1}$ defined by $L=E^{\perp}$ is totally geodesic if and only if $[A^2]=[-\Id]$, where $[A]$ is as in Theorem \ref{Amap} above.
\end{corollary}

\begin{proof}
Assume that $[A^2]=[-\Id]$.  Choose a representative $A \in [A]$ with unit-modulus eigenvalues.  Then  $\|Av\|=1$ and $A^2 v=-v$ for each $v \in S^{d-1}$.  The geodesic $c(t)=\cos(t)v+\sin(t)Av$ satisfies $\dot{c}(0) \in L_v$.  Then $\dot{c}(t) \in L_{c(t)}$ since $\dot{c}(t)=-\sin(t)v+\cos(t)Av=Ac(t)$, concluding the proof that $L$ is totally geodesic.

Conversely, assume that $L$ is totally geodesic.  Let $v \in S^{d-1}$ and choose a representative $A \in [A]$ satisfying $\|Av\|=1$.  The geodesic $c(t)=\cos(t)v+\sin(t)Av$ satisfies $\dot{c}(t) \in L_{c(t)}$ for each $t \in \R$.  Set $t=\frac{\pi}{2}$ and conclude that the $2$-plane spanned by $v$ and $Av$ is invariant under $A$.  As $\|Av\|=1$ and $A$ is skew-symmetric, $A^2v=-v$, concluding the proof.
\end{proof}

Let $\mathcal{X}=\{x \in S^{d-1}\,\vert\, E_x=T_xS^{d}\}$ denote the singular set for a  codimension one totally geodesic distribution $E$ on $S^{d-1}$.  Given a subset $U\subset S^{d-1}$, let $\Sigma(U)=\spn\{U\}\cap S^{d-1}$ denote the smallest totally geodesic subsphere of $S^{d-1}$ containing $U$.

\begin{lemma}\label{singular}
The singular set $\mathcal{X}$ satisfies $\Sigma(\mathcal{X})=\mathcal{X}$. 
\end{lemma}

\begin{proof}
There is nothing to prove if $\mathcal{X}=\emptyset$.  If $x \in \mathcal{X}$, then $-x \in \mathcal{X}$ since each great circle through $-x$ also passes through $x$.  It remains to prove that for linearly independent $x_1,x_2 \in \mathcal{X}$, the great circle $C_1:=\Sp(\{x_1,x_2\})\subset \mathcal{X}$.

If $x_3 \in C_1 \setminus \{\pm x_1,\pm x_2\}$, then the line $L_1:=T_{x_3}C_1$ is a subspace of $E_{x_3}$ since $x_1 \in \mathcal{X}$.  Let $L_2$ be any other line in $T_{x_3} S^{d-1}$ and let $C_2$ denote the great circle containing $x_3$ with tangent line $L_2$.  

Let $p \in C_2 \setminus \{\pm x_3\}$.  As $x_1,x_2 \in \mathcal{X}$ are linearly independent, the tangent lines at $p$ to the great circles in the totally geodesic $2$-sphere $\Sigma(C_1\cup C_2)$ that join $x_1$ to $p$ and $x_2$ to $p$ are transverse subspaces of $E_p\cap T_p \Sigma(C_1 \cup C_2)$.  Therefore $T_p\Sigma(C_1 \cup C_2) \subset E_p$.  In particular, the tangent line to $C_2$ at $p$ is a subspace of $E_p$, whence the line $L_2$ is a subspace of $E_{x_3}$, as required. 
\end{proof}

\begin{corollary}\label{basis}
The singular set $\mathcal{X}$ of a  codimension one totally geodesic distribution on $S^{d-1}$ does not contain a basis of $\R^{d}$.
\end{corollary}

The following simple lemma is applied to Riemannian exponential maps in subsequent sections.

\begin{lemma}\label{constant}
Let $E$ be a  codimension one totally geodesic distribution on $S^{d-1}$, $X$ a set, and $f:S^{d-1} \rightarrow X$ a function.  If $f$ is constant on curves everywhere tangent to $E$, then $f$ is constant.
\end{lemma}

\begin{proof}
Let $x \in S^{d-1}$.  The assumption implies that $f$ is constant on the union of geodesics with initial velocity in $E_x$, a totally geodesic subsphere of $S^{d-1}$ of codimension at most one.  Any two such subspheres intersect.
\end{proof}

\section{Proofs of Theorems \ref{thm:A}, \ref{thm:B}, and \ref{thm:C}}
Throughout this section $M$ denotes a complete $d$-dimensional Riemannian manifold with $\sec \geq 1$ and spherical rank at least $d-2$. Then $M$ is closed and has $\cvc(1)$.  In particular, for each $v \in SM$, the $1$-eigenspace $E_v$ of the Jacobi operator $\mathcal{J}_v$ (see Definition \ref{use}) is a nonempty subspace of $v^{\perp}$.

Recall that a point $p \in M$ is \textit{isotropic} if $\sec(\sigma)$ is independent of the $2$-plane section $\sigma \subset T_pM$ and \textit{nonisotropic} otherwise.  Hence, $p$ is an isotropic point if and only if $E_v=v^{\perp}$ for each $v \in S_pM$.  Let $\mathcal{I}$ and $\mathcal{O}$ denote the subsets of isotropic and nonisotropic points in $M$, respectively.  Note that $\mathcal{I}$ is closed in $M$ and that $\mathcal{O}$ is open in $M$.

\subsection{Preliminary structure and Proof of Theorem \ref{thm:A}.}\label{sec:3.1}
This subsection discusses a number of preliminary structural results that culminate in a proof of Theorem \ref{thm:A}.

For $p \in M$ and $v \in S_pM$, let $P_t: T_pM \rightarrow T_{\gamma_v(t)} M$ denote parallel translation along the geodesic $\gamma_v(t)$.  Define the subspace $D_v \subset v^{\perp}$ by 
\begin{eqnarray}\label{spheredist}
D_v &=& \spn\{w \in v^{\perp}\, \vert\, \sec(\dot{\gamma_v}(t), P_t w)=1\,\,\,\,\,\forall t \in \R\}\nonumber\\
 &=&\spn\{w \in v^{\perp}\, \vert\, P_t w\in E_{\dot{\gamma}(t)}\,\,\,\,\, \forall t \in \R\}\nonumber.
 \end{eqnarray}  Note that $D_v$ is a subspace of $E_v$ for each $v \in S_pM$. The spherical rank assumption implies $\dim(D_v) \geq d-2$.  In particular, the $1$-eigenspace distribution $E$ is a  codimension one totally geodesic (by Lemma \ref{totgeo}) distribution on $S_pM$ when $p \in \mathcal{O}$.
 
\begin{lemma}\label{obvious}
For each $v \in S_pM$
\begin{enumerate}
\item $D_v=D_{-v}$
\item If $w \in D_v$, then $\sec(\dot{\gamma_v}(t), P_t w)=1$ for all $t \in \R$.
\end{enumerate}
\end{lemma}

\begin{proof}
$(1)$ is immediate from the definition of $D_v$.  For (2), let $\{u_1,\ldots, u_k\}$ be a maximal collection of linearly independent vectors in 
\begin{equation}\nonumber
\{w \in v^{\perp}\, \vert\, P_t w \in E_{\dot{\gamma}(t)}\,\,\,\,\, \forall t\in \R\}
\end{equation} and express $w=\sum_{i=1}^{k} a_i u_i$. As $E_{\dot{\gamma}(t)}$ is a subspace, $P_t w=\sum_{i=1}^{k} a_i P_t u_i \in E_{\dot{\gamma}(t)}$ for each $t \in \R$, concluding the proof. 
\end{proof}

The rank of a vector $v \in S_pM$ is defined as $\dim(D_v)$.  The rank of a one dimensional linear subspace $L \leq T_pM$ is defined as the rank of a unit vector tangent to $L$.  The rank of a geodesic is the common rank of unit tangent vectors to the geodesic.

\begin{definition}
The \textit{spherical distribution} on $S_pM$, denoted by $D$, is the  tangent distribution defined by 
\begin{equation}\nonumber
S_pM \ni v \mapsto D_v \subset T_v(S_pM).\end{equation} Let $\mathcal{D}_p$ denote the subset of $S_pM$ consisting of rank $d-2$ vectors and let $\mathcal{D}=\cup_{p \in M} \mathcal{D}_p$ denote the collection of all rank $d-2$ unit vectors in $SM$. 
\end{definition}

As parallel translations along geodesics and sectional curvatures are continuous, the rank of vectors cannot decrease under taking limits.  This implies the following:

\begin{lemma}\label{cont}
For each $p \in M$, the regular set $\mathcal{D}_p$ is open and the spherical distribution $D$ on $S_pM$ is continuous on its regular set $\mathcal{D}_p$.
\end{lemma}

\begin{lemma}\label{same}
If $p$ is a nonisotropic point, i.e. $p \in \mathcal{O}$, then the spherical distribution $D$ and eigenspace distribution $E$ coincide on $S_pM$.
\end{lemma}

\begin{proof}
If not, then there exists a rank $d-2$ vector $v\in \mathcal{D}_p$ with the property that $E_v=T_v(S_pM)$.  Consider the codimension one totally geodesic subsphere $S \subset S_pM$ containing $v$ and determined by $T_vS=D_v$, namely $S=\exp_v(D_v)$.

Given $x \in S_pM\setminus S$ sufficiently close to $v$, let $C(v,x)$ denote the great circle through $v$ and $x$. Lemma \ref{cont} implies that the tangent line $T_x C(v,x)$ is transverse to the subspace $D_x$.  As $E$ is totally geodesic and $E_v=T_v(S_pM)$, it follows that $T_x C(v,x) \subset E_x$.  Conclude that $E_x=T_x(S_pM)$.  Lemma \ref{basis} implies that $E=T(S_pM)$, a contradiction since $p \in \mathcal{O}$.
\end{proof}

\begin{convention}
Parallel translation in $T_pM$ identifies the spherical distribution $D$ on $S_pM$ with a distribution defined on the tangent sphere $S(0, \pi) \subset T_pM$.  The latter is also denoted by $D$ when unambiguous.
\end{convention}


\begin{lemma}\label{spherjab}
Let $v \in S_pM$.  If $w\in D_v$, then $J(t)=\sin(t)P_t w$ is a Jacobi field along $\gamma_v(t)$.  In particular, $D_{\pi v} \subset K_{\pi v}$, where $K_{\pi v} = \ker(\dd(\exp_p)_{\pi v})$ (see \ref{kern}).
\end{lemma}

\begin{proof}
Lemma \ref{crit} and Lemma \ref{obvious}(2) imply  $J''(t)+R(J,\dot{\gamma}_v)\dot{\gamma}_v(t)=\sin(t)(-P_t w+R(P_t w,\dot{\gamma}_v)\dot{\gamma}_v(t))=0$ for all $t \in \R$.  As $J(\pi)=0$, $w \in K_{\pi v}$ by (\ref{jabfield}). 
\end{proof}

\begin{corollary}\label{point}
If $p \in \mathcal{O}$, then the restriction of $\exp_p$ to the tangent sphere $S(0,\pi)$ is a point map.
\end{corollary}

\begin{proof}
The map $\exp_p$ is constant on curves tangent to the kernel distribution defined by $S(0,\pi) \ni \pi v \mapsto K_{\pi v}\subset T_{\pi v} (S(0,\pi))$.  The distributions $E$ and $D$ coincide on $S_pM$ by Lemma \ref{same}.  Lemma \ref{spherjab} implies that $\exp_p$ is constant on curves tangent to the distribution $E$.  Lemma \ref{constant} implies the corollary.
\end{proof}

Let $\phi_t:SM \rightarrow SM$, $t\in \R$, denote the geodesic flow.  For $T>0$, let $$\fix_{T}=\{v \in SM \vert\, \phi_Tv=v\}.$$

\begin{lemma}\label{closed1}
If $p \in \mathcal{O}$, then $S_pM \subset \fix_{2\pi}$.
\end{lemma}

\begin{proof}
Corollary \ref{point} implies $\exp_p(S(0,\pi))=\{p'\}$ for some $p' \in M$.  The lemma is a consequence of the \textit{Claim}: for $v \in S_pM$, the geodesics $\gamma_v(t)$ and $\gamma_{-v}(t)$ satisfy $\dot{\gamma}_v(\pi)=-\dot{\gamma}_{-v}(\pi)$.

There exists a positive $\epsilon< \inj(M)$ such that $\gamma_v(\epsilon) \in \mathcal{O}$ since $\mathcal{O}$ is open in $M$.  Let $w=\dot{\gamma}_v(\epsilon)$.  Corollary \ref{point} implies that $q':=\gamma_w(\pi)=\gamma_{-w}(\pi)=\gamma_{-v}(\pi-\epsilon)$.  The geodesic segments $\gamma_{w}([\pi-\epsilon,\pi])$ and $\gamma_{-v}([\pi-\epsilon,\pi])$ each have length $\epsilon$ and meet at the points $p'$ and $q'$.  As $\epsilon < \inj(M)$, these segments coincide, implying the claim.
\end{proof}

\begin{lemma}\label{maxjab}
Let $v \in S_pM$ have rank $d-2$ and let $w$ be a unit vector in $v^{\perp} \cap D_v^{\perp}$.  The initially vanishing normal Jacobi field $J(t)$ along $\gamma_v(t)$  with $J(0)=0$ and $J'(0)=w$ has the form $J(t)=f(t)P_t w$ where $f(t)$ is the solution to the ODE $f''+\sec(P_t w,\dot{\gamma}_v)f=0$ with initial conditions $f(0)=0$ and $f'(0)=1$.
\end{lemma}

\begin{proof}
The initial conditions $f(0)=0$ and $f'(0)=1$ imply the initial conditions $J(0)=0$ and $J'(0)=w$.  The hypotheses and Lemma \ref{maximal} imply that $P_t w$ is an eigenvector of $\mathcal J_{\dot{\gamma}_v(t)}$ with eigenvalue $\sec(P_t w,\dot{\gamma}_v)(t)$.  Consequently, 
\begin{equation}\nonumber
J''(t)+R(J,\dot{\gamma}_v)\dot{\gamma_v}(t)=[f''(t)+\sec(P_t w,\dot{\gamma}_v(t))f(t)]P_t w=0
\end{equation} concluding the proof.
\end{proof}

\begin{corollary}\label{rank}
A vector $v \in S_pM$ has rank $d-1$ if and only if $\pi v \in \fc(p)$.
\end{corollary}

\begin{proof}
If $v$ has rank $d-1$, then Lemma \ref{spherjab} implies that $\pi v \in \fc(p)$.  If $v$ has rank $d-2$ then there is an initially vanishing Jacobi field of the form described by Lemma \ref{maxjab}.  The function $f(t)$ vanishes strictly before $\pi$ by the equality case of the Rauch Comparison Theorem \cite[Chapter 11]{doca}.
\end{proof}

Recall that $\fc(p)$ denotes the locus of first conjugate vectors in $T_pM$.

\begin{corollary}\label{cartan}
If there exists $p \in M$ with $\fc(p)=S(0,\pi)$, then $M=\mathcal{I}$, i.e. $M$ has constant curvatures equal to one.
\end{corollary}

\begin{proof}
Let $U_p = M\setminus \cut(p)$. By Corollary \ref{rank}, all vectors in $S_pM$ have rank $d-1$.  By Cartan's theorem on determination of the metric \cite[Theorem 2.1, pg. 157]{doca}, $U_p \subset \mathcal{I}$.  Therefore, $M=\cl({U}_p)\subset \mathcal{I}.$  
\end{proof}

\begin{lemma}\label{cartan2}
If $v \in S_pM$ has rank $d-1$ and the restriction of $\exp_p$ to $S(0, \pi)$ is a point map, then $M=\mathcal{I}$.
\end{lemma}

\begin{proof}
It suffices to prove $\fc(p)=S(0,\pi)$ by Corollary \ref{cartan}.  Let $X=\fc(p)\cap S(0,\pi)$. The vector $\pi v \in X$ by Corollary \ref{rank}; therefore $X$ is a nonempty subset of $S(0,\pi)$.  The subset $X$ is closed in $S(0,\pi)$ by Lemma \ref{closed}.  It remains to demonstrate that $X$ is an open subset of $S(0,\pi)$.

This fails only if there exists $x \in X$ and a sequence $x_i \in S(0, \pi)\setminus X$ converging to $x$.  As $\exp_p$ is a point map on $S(0, \pi)$ each $x_i$ is a conjugate vector.  As $x_i \notin \fc(p)$ there exists $s_i \in (0,1)$ such that $s_i x_i \in \fc(p)$.  By Lemma \ref{maxjab}, there exist Jacobi field $J_i(t)=f_i(t)P_t w_i$ along the geodesics $\gamma_{\mathbf{r}(x_i)}(t)$ with $f_i(0)=f_i(s_i)=f_i(\pi)=0$ for each index $i$.  Note that $\min\{s_i,\pi-s_i\}> \inj(M)/2$.  Therefore, $s_i x_i$ converge to a conjugate vector $s x$ with $0<s<1$, a contradiction.
\end{proof}

\begin{proposition}\label{constantrank}
$SM=\mathcal{D}$ or $M=\mathcal{I}$. 
\end{proposition}

\begin{proof}
Assume that $\mathcal{I}$ is a proper subset of $M$, or equivalently, that $\mathcal{O} \neq \emptyset$.  Corollary \ref{point} and Lemma \ref{cartan2} imply $\mathcal{D}_p=S_pM $ for each $p \in \mathcal{O}$. Therefore $\mathcal{D} \neq \emptyset$.  As $\mathcal{D}$ is an open subset of the connected $SM$, it remains to prove that $\mathcal{D}$ is a closed subset of $SM$.

This fails only if there exists a sequence of rank $d-2$ vectors $v_i \in \mathcal{D}$ with $v_i$ converging to a vector $v \in SM$ of rank $d-1$.  Lemma \ref{closed1} implies each of the geodesics $\gamma_{v_i}$ is closed and has $2\pi$ as a period; therefore, $\gamma_v$ is a closed geodesic having $2\pi$ as a period.  Let $p_i \in M$ denote the footpoint of each $v_i$ and $p \in M$ the footpoint of $v$.  As the rank of $v_i$ is $d-2$, the geodesic $\gamma_{v_i}$ enters $\mathcal{O}$ at some time $t_i$.  Replace $v_i$ with $w_i=\dot{\gamma}_{v_i}(t_i)$. After possibly passing to a subsequence, the sequence of rank $n-2$ vectors $w_i$ with footpoints $q_i \in \mathcal{O}$ converge to a rank $d-1$ vector $w$ with footpoint $q$.

Continuity of $\exp:TM \rightarrow M$ and Lemma \ref{point} imply that $\exp_q$ restricts to a point map on the tangent sphere  $S(0,\pi) \subset T_qM$.  Lemma \ref{cartan2} implies $M=\mathcal{I}$, a contradiction. 
\end{proof}

\medskip

\noindent \textit{Proof of Theorem \ref{thm:A}}:
Seeking a contradiction, assume that $M \neq \mathcal{I}$.  Then  $SM=\mathcal{D}$ by Proposition \ref{constantrank}.  For $p \in M$, the spherical distribution $D$ is a nonsingular codimension one tangent distribution on $S_pM$, an \textit{even dimensional} sphere since $M$ is odd dimensional.  This distribution is continuous by Lemma \ref{cont}, a contradiction. 
\hfill $\Box$






\subsection{Adapted Frames}\label{sec:adapt}
This subsection consists of preliminary results that will culminate in the proof of Theorem \ref{thm:B} in the next subsection.  If $M$ does not have constant curvatures equal to one, then Theorem \ref{thm:A} implies $d=\dim(M)$ is even and Proposition \ref{constantrank} implies every tangent vector has rank $d-2$ ($SM=\mathcal{D}$). These are \textit{standing assumptions} on $M$ throughout this subsection.  The main result is the following proposition; its proof appears at the end of this subsection.

\begin{proposition}\label{noisotropic}
If $M$ does not have constant curvatures equal to one, then $M$ has no isotropic points ($M=\mathcal{O}$).
\end{proposition}

\begin{lemma}\label{lowrank}
For each $p \in M$, the spherical distribution $D$ is a smooth tangent distribution on $S_pM$.
\end{lemma}

\begin{proof}
It suffices to prove smoothness of $D$ on a metric ball $B$ contained in the tangent sphere $S_pM$.  As the center $b_0$ of  $B$ is a rank $d-2$ vector, there exists a unit vector $w\in b_{0}^{\perp}$ and a $t_0>0$ such that $\sec(\dot{\gamma}_{b_0}(t_0),P_{t_0} w)> 1$. Therefore $\gamma_{b_0}(t_0) \in \mathcal{O}$, and since $\mathcal{O}$ is open,  $\gamma_{b}(t_0) \in \mathcal{O}$ for all $b \in B$ after possibly reducing the radius of $B$.

Lemma \ref{same} implies $D_{\dot{\gamma}_b(t_0)}=E_{\dot{\gamma}_b(t_0)}$ for each $b \in B$.  The unit tangent vectors $\dot{\gamma}_b(t_0)$ vary smoothly with $b \in B$. Remark \ref{smoothremark} implies $D_{\dot{\gamma}_b(t_0)}$ varies smoothly with $b\in B$.  The lemma follows since $D_b$ is obtained by parallel translating along $\gamma_b$ for time $t_0$ the subspace $D_{\dot{\gamma}_b(t_0)}$ to $T_b (S_pM)$.
\end{proof}

The proof of Proposition \ref{noisotropic} is based on a curvature calculation in special framings along geodesics.  To introduce these framings, let $p \in M$, $v \in S_pM$, and let $\{e_1,\ldots,e_{d-1}\} \subset T_v (S_pM)$ be an orthonormal basis with $e_1,\ldots,e_{d-2} \in D_v$. Define $E_0(t)=P_t v=\dot{\gamma}_v(t)$ for $t>0$ and $E_i(t)=P_t e_i$ for $i \in \{1,\ldots,d-1\}$ and $t>0$.

\begin{definition}
The parallel orthonormal framing $\{E_0(t),\ldots,E_{d-1}(t)\}$ along the ray $\gamma_v:[0,\infty)\rightarrow M$ is an \textit{adapted framing}.
\end{definition}

The following describes curvature calculations in polar coordinates using adapted framings.

Suppose that $B \subset S_pM$ is a metric ball of radius less than $\pi$.  Then $TB$ is trivial and the restriction of the spherical distribution $D$ to $B$ is trivial.  By Lemma \ref{lowrank}, there are \textit{smooth} unit length vector fields $e_1,\ldots,e_{d-2}$ on $B$ tangent to $D$.  An orientation on $S_pM$ determines a positively oriented orthonormal framing $\{e_1,\ldots,e_{d-1}\}$ of $B$.  For each $b \in B$, let $\{E_0(t),\ldots,E_{d-1}(t)\}$ be the associated adapted framing along the ray $\gamma_b$.  

Now fix $v \in B$.  For $T>0$ such that $Tv$ is not a conjugate vector, $\exp_p$ carries a neighborhood $U$ of $Tv$ in $T_pM$  diffeomorphically onto a neighborhood $V$ of $\exp_p(Tv)$ in $M$.  After possibly reducing the radius of $B$, the radial retraction of $U$ to the unit sphere $\textbf{r}(U)$ coincides with $B$.  The collection of adapted framings along geodesic rays with initial tangent vectors in $B$ restrict to an orthonormal framing $\{E_0,\ldots,E_{d-1}\}$ of the open set $V$ in $M$.   To calculate the Christoffel symbols in this framing, first define $a_{ij}^k:B \rightarrow \R$ by
\begin{equation}\label{structure}
[e_i,e_j]=\sum_{k=1}^{d-1}a_{ij}^k e_k.
\end{equation} 
As $Tv$ is not a conjugate vector, the geodesic spheres $S(p,t)$ with center $p$ and radius $t$ close to $T$ intersect the neighborhood $V$ in smooth codimension one submanifolds. The vector fields $E_1(t),\ldots,E_{d-1}(t)$ are tangent to the distance sphere $S(p,t)$ in $V$ and have outward pointing unit normal vector field $E_0(t)$.  In what follows, $g':=E_{0}(g)$ denotes the radial derivative of a function $g$.

For each unit speed geodesic $\gamma(t)$ with initial velocity vector in $B$, let $J_i(t)$ denote the Jacobi field along $\gamma$ with initial conditions $J_i(0)=0$ and $J_{i}'(0)=e_i\in T_{\dot{\gamma}(0)}(S_pM)$.  Lemmas \ref{spherjab} and \ref{maxjab} imply 
\begin{eqnarray}\label{again}
&J_i(t)=\sin(t)E_i(t),& i\in\{1,\ldots, d-2\},\nonumber\\ \\
&J_{d-1}(t)=f(t)E_{d-1}(t),&\nonumber
\end{eqnarray}
where $f(t)$ is the solution of the ODE 
\begin{equation}\nonumber
f''+\sec(E_{0},E_{d-1})f=0, \quad \text{ with }\quad f(0)=0, \, f'(0)=1.
\end{equation} 
For $t$ close to $T$, define $F_t: B \rightarrow M$ by $F_t(b)=\exp_p(tb)$.  The chain rule and (\ref{jabfield}) imply
\begin{equation}\label{related}
 \dd F_t(e_i)=J_i(t)
\end{equation}  for $i \in \{1,2, \ldots, d-1\}$.  Use (\ref{structure}), (\ref{related}), and the fact that the Jacobi fields $J_i$ are invariant under the radial (geodesic) flow generated by $E_0$ to deduce
\begin{align}\label{brackets1}
&[J_i,J_j]=\sum_{k=1}^{d-1} a_{ij}^k J_k,&  
&\mathcal{L}_{E_0} J_i=[E_0,J_i]=0.&      
\end{align}
Use (\ref{again}) and (\ref{brackets1}) to calculate that for $i,j \in \{1,\ldots, d-2\}$:
\begin{align}\label{brack}
&[E_0,E_i] =-\cot E_i&\nonumber\\\nonumber \\
&[E_0,E_{d-1}]=\tfrac{-f'}{f}E_{d-1}\nonumber&\\\nonumber \\
&[E_i,E_j]=\sum_{k=1}^{d-2} \tfrac{a_{ij}^k}{\sin}E_k+\tfrac{a_{ij}^{d-1}f}{\sin^2}E_{d-1}&\\\nonumber \\
&[E_i,E_{d-1}] = \sum_{k=1}^{d-2} \tfrac{a_{i\,d-1}^k}{f}E_k+(\tfrac{a_{i\, d-1}^{d-1}}{\sin}-\tfrac{E_i(f)}{f})E_{d-1}.&\nonumber
\end{align}
\begin{lemma}\label{christoffel1}
Let $i,j \in \{1,\ldots, d-2\}$.  The orthonormal framing $\{E_0,\ldots, E_{d-1}\}$ has Christoffel symbols given by $\nabla_{E_0} E_k=0$ for each $k \in \{0,\ldots,d-2\}$ and:
\begin{align*}
&\nabla_{E_i} E_0=\cot E_i&\\ \\ 
&\nabla_{E_i} E_j=-\cot \delta_{i}^{j} E_0+\sum_{k=1}^{d-2}\tfrac{a_{ij}^{k}-a_{jk}^{i}+a_{ki}^{j}}{2\sin}E_k-\tfrac{1}{2}\{\tfrac{a_{i\, d-1}^{j}+a_{j\, d-1}^{i}}{f}+\tfrac{a_{ji}^{d-1}f}{\sin^2}\}E_{d-1}&\\ \\
&\nabla_{E_i} E_{d-1}=\sum_{k=1}^{d-2} \tfrac{1}{2}\{\tfrac{a_{i\, d-1}^{k}+a_{k\, d-1}^{i}}{f}+\tfrac{a_{ki}^{d-1}f}{\sin^2}\}E_k&\\ \\
&\nabla_{E_{d-1}} E_0= \tfrac{f'}{f}E_{d-1}&\\ \\
&\nabla_{E_{d-1}} E_j=\sum_{k=1}^{d-2} \tfrac{1}{2}\{\tfrac{a_{k\, d-1}^{j}-a_{j\, d-1}^{k}}{f}-\tfrac{a_{jk}^{d-1} f}{\sin^2}\}E_k+(\tfrac{E_j(f)}{f}-\tfrac{a_{j\, d-1}^{d-1}}{\sin})E_{d-1}&\\ \\
&\nabla_{E_{d-1}} E_{d-1}=-\tfrac{f'}{f}E_0+\sum_{k=1}^{d-2} (\tfrac{a_{k\, d-1}^{d-1}}{\sin}-\tfrac{E_k(f)}{f})E_k& 
\end{align*}

where $f(t)$ is the solution of the ODE
\begin{equation*}
f''+\sec(E_0,E_{d-1})f=0, \quad \text{ with }\quad f(0)=0, \, f'(0)=1.
\end{equation*}
\end{lemma}

\begin{proof}
Calculate using (\ref{koszul}) and (\ref{brack}). 
\end{proof}
Use Lemma \ref{christoffel1} to derive the curvature components:  For $i,j \in \{1, \ldots, d-2\}$, 
\begin{equation}\label{curv1}
R(E_i,E_j,E_0,E_{d-1})=\frac{-(a_{ij}^{d-1} f \csc)'}{\sin}
\end{equation}
\begin{equation}\label{curv2}
R(E_{d-1},E_0,E_j,E_{d-1})=\frac{(a_{j\, d-1}^{d-1}f \csc-E_j(f))'}{f}.
\end{equation}
\medskip

\noindent \textit{Proof of Proposition \ref{noisotropic}}:
The goal is to prove $\mathcal{I}=M$ or $\mathcal{I}=\emptyset$.  The set of isotropic points $\mathcal{I}$ is closed in $M$ and $M$ is connected.  It suffices to prove that $\mathcal{I}$ is open in $M$.  Let $p \in \mathcal{I}$ and $v \in S_pM$.  As $v$ has rank $d-2$, there exists a positive $s<\pi$ such that $q:=\gamma_v(s)$ is the first conjugate point to $p$ along the geodesic $\gamma_v(t)$.

\textit{Claim}: $\mathcal{I}$ contains an open neighborhood of $q$ in $M$.  

Assuming the claim holds, $\mathcal{I}$ contains an open neighborhood of the point $p$ in $M$ since the property of being a first conjugate point along a geodesic is symmetric.  Hence $\mathcal{I}$ is open in $M$.

\textit{Proof of Claim}: Let $w=-\dot{\gamma}_v(s)$ and note that $p=\gamma_w(s)$.  Let $B$ be a small metric ball in $S_qM$ containing $w$ and trivialize the tangent bundle of $B$  with orthonormal vector fields $\{e_1,\ldots, e_{d-1}\}$ with $e_i(b) \in D_b$ for each $b \in B$ and $i \in \{1,\ldots, d-2\}$. Consider the induced adapted framings $\{E_0, \ldots, E_{d-1}\}$ along geodesics with initial velocity vectors in $B$.

If $q$ is not contained in an open neighborhood of isotropic points, then there exists a sequence $q_i \in \mathcal{O}$ converging to $q$.  As all vectors have rank $d-2$ the spherical distributions on $S_{q_i}M$ converge to the spherical distribution on $S_qM$.  

As $q_i \in \mathcal{O}$, Lemma \ref{same} implies that the spherical distribution on each $S_{q_i}M$ is totally geodesic.  Therefore, the limiting spherical distribution on $S_qM$ is totally geodesic.  By Corollary \ref{contact}, the limiting distribution on $S_qM$ is a contact distribution.  In particular, the function
\begin{equation}\nonumber
a_{12}^{d-1}=\langle [e_1,e_2], e_{d-1} \rangle
\end{equation} is nonzero on $B$.  Use (\ref{curv1}) to calculate 
\begin{equation}\label{conjeqn}
R(E_1,E_2,E_0,E_{d-1})(t)=\tfrac{a_{12}^{d-1}(w)}{\sin^3(t)}(\cos(t)f(t)-\sin(t)f'(t))
\end{equation} for $t \in (0,s)$ along $\gamma_w(t)$.  

As $p \in \mathcal{I}$, the curvature tensor vanishes on orthonormal $4$-frames at the point $p.$  Therefore as $t$ converges to $s$, the left hand side of (\ref{conjeqn}) converges to zero.  As $a_{12}^{d-1}$ is nonzero on $B$, $(\cos f-\sin f')\rightarrow 0$  as $t\rightarrow s$.

Only the Jacobi field $J_{d-1}(t)$ can vanish before time $\pi$.  As $p$ is conjugate to $q$, $f(t)\rightarrow 0$ as $t\rightarrow s$.  As $s< \pi$, $\sin(s)\neq 0$. Conclude that $f(s)=f'(s)=0$, a contradiction since $J_{d-1}(t)=f E_{d-1}(t)$ is a nonzero Jacobi field along $\gamma_w(t)$.
\hfill $\Box$

\subsection{Proof of Theorem B}\label{sec:B}\hfill

\bigskip


\noindent \textit{Proof of (1)}:
Let $v \in S_pM$.  Since every tangent vector has rank $d-2$, $\dim(D_v)=\dim(v^{\perp})-1$.  Proposition \ref{noisotropic} and Lemma \ref{same} imply $D_v=E_v$.  Lemma \ref{crit} concludes the proof. 
\hfill $\Box$

\bigskip

\noindent \textit{Proof of (2)}:
Proposition \ref{noisotropic} and Lemma \ref{closed1} imply that $SM \subset \fix_{2\pi}$.  
\hfill $\Box$

\bigskip

\noindent \textit{Proof of (3)}:
As in the proof of (1), $D_v=E_v$ and $\dim(D_v)=\dim(v^{\perp})-1=d-2$ for all $v \in SM$.  Lemma \ref{totgeo} implies that for each $p \in M$, the eigenspace distribution $E$ on $S_pM$ is a nonsingular codimension one totally geodesic distribution.  Theorem \ref{Amap} yields a nonsingular projective class $[A_p] \in PGL(T_pM)$ for each $p \in M$, varying smoothly with $p \in M$ by Remark  \ref{smoothremark}.  For each $p \in M$ there are precisely two representatives of the projective class $[A_p]$ having determinant one.  As $M$ is simply connected there exists a smooth section $p \mapsto A_p \in [A_p]$.  Item (3) is therefore a consequence of the polar decomposition of $A_p$, see \cite[Lemma 2.32, pg. 64]{bes} for details.
\hfill $\Box$

\bigskip

The proof of item (4) of Theorem \ref{thm:B} requires some preliminary lemmas.\\

Corollary \ref{point} and Proposition \ref{noisotropic} imply that the restriction of  $\exp_p$ to the tangent sphere $S(0,\pi) \subset T_pM$ is a point map for each $p \in M$.  Define the map $F:M \rightarrow M$ by $F(p)=\exp_p(S(0,\pi))$.  Then $F^2=\Id$ by item (2) of Theorem \ref{thm:B}.

\begin{lemma}\label{isometry}
$F$ is an isometry of $M$.
\end{lemma}

\begin{proof}
The map $F$ sends each complete geodesic in $M$ into itself while preserving the lengths of subsegments.
\end{proof}

\begin{lemma}\label{fixedpoint}
If $F$ has a fixed point, then $F=\Id$.
\end{lemma}

\begin{proof}
 By Lemma \ref{isometry}, it suffices to prove if $F(p)=p$, then the derivative map $\dd F_p=\Id$.  The eigenvalues of the derivative map $\dd F_p$ are square roots of unity since $F^2=\Id$.   If $v \in T_pM$ is a unit  length eigenvector of eigenvalue $-1$, then $\dd F_p(v)=\dot{\gamma}_v(\pi)=-v$.  Therefore, $\gamma_{v}(\pi+t)=\gamma_{v}(-t)$ for all $t$.  By the chain rule, $\dot{\gamma}_v(\pi+t)=-\dot{\gamma}_v(-t)$ for all $t$.  When $t=-\frac{\pi}{2}$ this implies $\dot{\gamma}_v(\frac{\pi}{2})=-\dot{\gamma}_v(\frac{\pi}{2})$, a contradiction.
\end{proof}

\begin{lemma}\label{uppersec}
If $\sec<9$, then $F$ has a fixed point.
\end{lemma}

\begin{proof}
If $F$ has no fixed points, then the displacement function of $F$, $x \mapsto d(x,F(x))$, obtains a positive minimum value at some $p \in M$ as $M$ is compact.  A minimizing geodesic segment $\gamma$ that joints $p$ to $F(p)$ has length $L\leq \diam(M)<\pi$ by Toponogov's diameter rigidity theorem \cite{top} (see also \cite[Remark 3.6, pg. 157]{sa}).  Let $m$ denote the midpoint of the segment $\gamma$.  The union $\gamma \cup F(\gamma)$ forms a smoothly closed geodesic of length $2L$ since otherwise $d(m,F(m))<L=d(p,F(p)$.  By item (2) and since $F$ has no fixed points, $2L \in \{2\pi/(2k+1)\, \vert\, k\geq 1\}$.  Therefore, $\inj(M)\leq L \leq \pi/3$.  As $M$ is simply connected, even dimensional, and positively curved, $\inj(M)=\conj(M)$.  The Rauch comparison theorem and the assumption $\sec<9$ imply that $\conj(M)> \pi/3$, a contradiction.
\end{proof}

\noindent \textit{Proof of (4)}:
Lemmas \ref{fixedpoint} and \ref{uppersec} imply that $F=\Id$.  It follows that each geodesic in $M$ is a closed geodesic having $\pi$ as a period.  If a closed geodesic of length $\pi$ is not simple, then there exist a geodesic loop in $M$ of length at most $\pi/2$.  In this case, $\inj(M)\leq \pi/4$, contradicting $\inj(M)=\conj(M)>\pi/3$.  Therefore, each geodesic in $M$ is simple, closed, and of length $\pi$.

Each unit speed geodesic starting at a point $p \in M$ of length $\pi$ has equal index $k=1,3,7,$ or $\dim(M)-1$ in the pointed loop space $\Omega(p,p)$ by the Bott-Samelson Theorem \cite[Theorem 7.23]{bes}.  The multiplicity of each conjugate point to $p$ in the interior of these geodesics is one since the spherical Jacobi fields defined in Lemma \ref{spherjab} do not vanish before time $\pi$.  If $k \geq 3$, the Jacobi field given by Lemma \ref{maxjab} has a pair of consecutive vanishing times $0<t_1<t_2<\pi$ satisfying $t_2-t_1 \leq \pi/k \leq \pi/3$.  This contradicts $\conj(M)>\pi/3$ as $\sec<9$.  Conclude that $k=1$ and that $M$ has the homotopy type of $\mathbb{C}\mathbb{P}^{d/2}$ by \cite[Theorem 7.23]{bes}.
\hfill $\Box$

\subsection{Proof of Theorem \ref{thm:C}}\label{sec:C}

Recall that a Riemannian manifold satisfies the \textit{Raki\'c duality principle} if for each $p \in M$, orthonormal vectors $v,w \in S_pM$, and $\lambda \in \R$,  $v$ is a $\lambda$-eigenvector of the Jacobi operator $\mathcal{J}_w$ if and only if $w$ is a $\lambda$-eigenvector of the Jacobi operator $\mathcal{J}_v$.  This subsection contains the proof of Theorem \ref{thm:C}, an easy consequence of the next proposition.

\begin{proposition}\label{rakic}
Let $M$ be a complete and simply connected Riemannian $d$-manifold with $d\geq 4$ even, $\sec \geq 1$, spherical rank at least $d-2$, and no isotropic points.  If $M$ satisfies the Raki\'c duality principle, then $M$ is isometric to $\mathbb{C}\mathbb{P}^{d/2}$ endowed with the symmetric metric having constant holomorphic curvatures equal to $4$.
\end{proposition}

The proof of this proposition appears at the end of the subsection.  As a preliminary step, observe that the proof of item (3) of Theorem \ref{thm:B} shows that there exists a smooth section $p \mapsto A_p \in SL(T_pM)$ where each $A_p$ is skew-symmetric and satisfies $D_v=\spn\{v,A_pv\}^{\perp}$ for each $v \in S_pM$.  Define $\lambda:SM \rightarrow \R$ by $\lambda(v)=\sec(v,A_pv)$ where $p$ denotes the footpoint of the vector $v\in SM$.

\begin{lemma}\label{Asquared}
$A_p^2=-\Id$ for each $p \in M$.
\end{lemma}

\begin{proof}
The proof of item (1) of Theorem \ref{thm:B} shows that $A_pv$ is orthogonal to the $1$-eigenspace $D_v$ of the Jacobi operator $\mathcal{J}_v$.  Therefore $\lambda(v)>1$ and $A_pv/\|A_pv\|$ is a unit vector in the $\lambda(v)$-eigenspace of $\mathcal{J}_v$.  Similarly, $\lambda(A_pv/\|A_pv\|)>1$ and $A^2_p v/\|A_p^2v\|$ is a unit vector in the $\lambda(A_pv/\|A_pv\|)$-eigenspace of the Jacobi operator $\mathcal{J}_{A_pv/\|A_pv\|}$.  The Raki\'c duality property implies that $v$ is a unit vector in the $\lambda(v)$-eigenspace of the Jacobi operator $\mathcal{J}_{A_pv/\|A_pv\|}$.  The Jacobi operator $\mathcal{J}_{A_pv/\|A_pv\|}$ has two eigenspaces, the $1$-eigenspace $D_{A_pv/\|A_pv\|}$ of dimension $d-2$ and its one dimensional orthogonal complement, the $\lambda(A_pv/\|A_pv\|)$ eigenspace.  Conclude that for each $v \in S_pM$, $\lambda(v)=\lambda(A_pv/\|A_pv\|)$ and by skew-symmetry of $A_p$ that $v=-A_p^2v/\|A^2_pv\|$.  As $A_p^2v$ is a multiple of $v$ for each $v \in S_pM$ and $A_p$ is skew-symmetric of determinant one, $A_p^2=-\Id$.
\end{proof}

Fix $p \in M$ and a metric ball $B$ in the tangent sphere $S_pM$.  Let $\{e_1,\ldots,e_{d-1}\}$ be a smooth framing of $B$ with $\{e_1,\ldots,e_{d-2}\}$ tangent to the spherical distribution $D$.  

\begin{lemma}\label{maxdirection}
The field $e_{d-1}$ satisfies $\nabla_{e_{d-1}} e_{d-1}=0$ on $B$ with respect to the round metric on $S_pM$.  Equivalently, $a_{j\, d-1}^{d-1}=0$ for each $j \in \{1, \ldots, d-2\}$.
\end{lemma}

\begin{proof}
The first assertion is a consequence of Corollary \ref{geodesic} and Lemma \ref{Asquared}.  The second is derived from (\ref{koszul})  
$$2\langle \nabla_{e_{d-1}} e_{d-1}, e_j \rangle=\langle [e_{d-1},e_{d-1}], e_j \rangle-\langle[e_{d-1},e_j],e_{d-1} \rangle+\langle[e_j,e_{d-1}], e_{d-1} \rangle=2a_{j\, d-1}^{d-1}.$$
\end{proof}

Consider the \textit{adapted framing} $\{E_0(t),\ldots,E_{d-1}(t)\}$ along geodesics with initial velocity in $B$ induced by the framing $\{e_1,\ldots,e_{d-1}\}$ of $B$.  Let $\epsilon< \inj(M)$ and let $J(b,t)=f(b,t)E_{d-1}(t)$ be the Jacobi field along $\gamma_b(t)$ defined by Lemma \ref{maxjab}. Then $f>0$ on $B \times(0,\epsilon)$.

\begin{proposition}\label{radial}
The function $f:B \times(0, \epsilon) \rightarrow \R$ is radial: $E_j(f)=0$ for each $j\in \{1,\ldots,d-1\}$, or equivalently, $f(b,t)$ does not depend on $b\in B$.
\end{proposition}

\begin{proof}
Lemma \ref{maxdirection} and (\ref{curv2}) imply $R(E_{d-1},E_0,E_j,E_{d-1})=\frac{-E_j(f)'}{f}$ for each $j \in \{1,\ldots, n-2\}$. For each $b \in B$ and $t\in (0, \epsilon)$, $E_{d-1}(b,t)$ is an eigenvector of eigenvalue $\lambda(E_0(b,t))$ for the Jacobi operator $\mathcal{J}_{E_0(b,t)}$.  The symmetry property implies that $E_{0}(b,t)$ is an eigenvector of the Jacobi operator $\mathcal{J}_{E_{d-1}(b,t)}$.  Conclude $E_j(f)'=0$ for each $j \in \{1,\ldots, d-2\}$.  Use (\ref{brack}) to calculate

\begin{eqnarray}
0&=&E_j(f)'\nonumber\\
&=&E_0 E_j(f)\nonumber\\
&=&[E_0,E_j](f)+E_j E_0(f)\nonumber\\
&=&-\cot E_j(f)+E_{j}(f')\nonumber\\
&=&E_j(f'-\cot f).\nonumber
\end{eqnarray}

Let $g=f'-\cot f$. Corollary \ref{contact} and the fact that the time $t$-map of the radial flow generated by $E_0$ carries the spherical distribution $D$ to the distribution spanned by $\{E_1(t),\ldots,E_{d-2}(t)\}$ on $\exp_p(tB)\subset S(p,t)$ imply that the latter distribution is contact.  Conclude that $E_{d-1}(g)=0$ and that $g$ is a radial function.

Therefore $$h:=\frac{g}{\sin}=\frac{f'\sin-\cos f}{\sin^2}=(\frac{f}{\sin})'$$ is a radial function.  Let $k=\frac{f}{\sin}$ and consider the restriction $k(t)$ to a geodesic $\gamma_b(t)$ with $b\in B$.  By L'Hopital's rule and the initial condition $f'(0)=1$, $\lim_{t \rightarrow 0} k(t)=\frac{f'(0)}{\cos(0)}=1$.  By the fundamental theorem of calculus, $k(t)=1+\int_{0}^{t} h(s)\,  ds$ is a radial function.  Therefore $f=k \sin$ is a radial function. 
\end{proof}

\medskip

\noindent \textit{Proof of Proposition \ref{rakic}}:
It suffices to prove that $\lambda:SM \rightarrow \R$ is constant by \cite[Theorem 2, pg. 193]{chi}.  Fix $p \in M$ and a metric ball $B \subset S_pM$ as in Proposition \ref{radial}.  Proposition \ref{radial} implies that $\lambda$ is constant on $B$ since by the Jacobi equation,  $\lambda(b)=\lim_{t \rightarrow 0} \frac{-f''}{f}(b,t)$ for each $b \in B$.  As $S_pM$ is connected, $\lambda:S_pM \rightarrow \R$ has a constant value $\lambda(p)$.  Each point $p\in M$ is an Einstein point with $\Ric_p=(\lambda(p)+d-2)\g_p$.  The adaptation of Schur's Theorem for Ricci curvatures \cite[Note 3, Theorem 1, pg. 292]{kono} implies that $M$ is globally Einstein. Therefore $\lambda(p)$ is independent of $p \in M$.
\hfill $\Box$

\medskip

\noindent \textit{Proof of Theorem \ref{thm:C}}:
Apply Theorem \ref{thm:A}, Proposition \ref{noisotropic}, and Proposition \ref{rakic}.
\hfill $\Box$







\section{Proof of Theorem \ref{thm:D} in real dimension at least six}\label{sec:D1}

Throughout this section, $M$ is K\"ahlerian with complex structure $J:TM \rightarrow TM$, real even dimension $d \geq4$, $\sec \geq 1$, and spherical rank at least $d-2$.  This section contains preliminary results, culminating in the proof of Theorem \ref{thm:D} when $d \geq 6$.

As $M$ is orientable (complex), even-dimensional, and positively curved, $M$ is simply connected by Synge's theorem.  As $M$ is K\"ahlerian, its second betti number $b_2(M) \neq 0$, whence $M$ is not homeomorphic to a sphere.  Therefore $M$ does not have constant sectional curvatures.

 Proposition \ref{noisotropic} now implies that $M$ has no isotropic points ($M=\mathcal{O}$).  Proposition \ref{constantrank} implies that every vector in $M$ has rank $d-2$.  Lemmas \ref{totgeo} and \ref{same} imply that that the eigenspace distribution is a nonsingular codimension one distribution on each unit tangent sphere in $M$.   By Theorem \ref{Amap}, there exists a nonsingular projective class $[A_p] \in PGL(T_pM)$ of skew-symmetric maps such that $D_v=E_v=\{v,A_pv\}^{\perp}$ for each $p \in M$ and $v \in S_pM$.



\subsection{Relating the complex structure and the eigenspace distribution.}
Fix $p \in M$ and choose a representative $A_p \in [A_p]$.  Assume that $V=\sigma_1 \oplus \sigma_2$ is an orthogonal direct sum of two $A_p$-invariant $2$-plane sections.  There exist scalars $0<\mu_1$ and $0<\mu_2$ such that $\|A_p v_i\|=\mu_i$ for each unit vector $v_i \in \sigma_i$.  There is no loss in generality in assuming $\mu_1 \leq \mu_2$ and if equality $\mu_1=\mu_2$ holds, then $\lambda_1 \leq \lambda_2$.


For a unit vector $v \in S_pM$, let $\lambda(v)=\sec(v,A_pv)$.  Then $A_pv$ is an eigenvector of the Jacobi operator $\mathcal{J}_v$ with eigenvalue $\lambda(v)>1$. Note that $\lambda(v)$ is the maximal curvature of a $2$-plane section containing the vector $v$.  Therefore, $\lambda(A_pv/\|A_pv\|)\geq \lambda(v)$ with equality only if $A_p^2v$ and $v$ are linearly dependent.  For a vector $v_i \in \sigma_i$, let $\bar{v}_i=A(v_i)/\mu_i$.  With this notation, $\bar{\bar{v_i}}=-v_i$.

\begin{lemma}\label{c0}
Assume that $\{u,v,w\}\subset V$ are orthonormal vectors with $u,v \in \sigma_i$ and $w \in \sigma_j$ with $i\neq j \in \{1,2\}$.  Then $R(u,v,w,u)=0$ and $R(u,w,w,u)=1$.  
\end{lemma}

\begin{proof}
As $u\in \sigma_i$, an $A_p$-invariant $2$-plane, the orthogonal $2$-plane $\sigma_j$ is contained in $E_u$.  In particular, $w \in E_u$, implying the lemma.
\end{proof}


\begin{lemma}\label{c1}
Let $v_i \in \sigma_i$, $i=1,2$, be unit-vectors.  If $v=a v_1+b v_2$ is a unit-vector, then $\lambda(v)=a^2 \lambda_1+b^2 \lambda_2$.
\end{lemma}

\begin{proof}
Observe that $\lambda(v)\|A_p v\|^2=R(v,A_p v,A_p v,v)$ or equivalently
 $$\lambda(v)(a^2 \mu_1^2+b^2 \mu_2^2)=R(a v_1 +b v_2,a \mu_1 \bar{v}_1+b \mu_2 \bar{v}_2 , a \mu_1 \bar{v}_1+b \mu_2 \bar{v}_2,a v_1 +b v_2).$$  Expanding the above, using Lemma \ref{c0}, and simplifying yields

\begin{equation}\label{1}
\lambda(v)(a^2\mu_1^2+b^2\mu_2^2)=a^4\mu_1^2\lambda_1+a^2b^2(\mu_1^2+\mu_2^2)+b^4\mu_2^2\lambda_2+\Phi
\end{equation} 
where

\begin{equation}\label{2}
\Phi=2a^2b^2\mu_1\mu_2[R(v_1,\bar{v}_1,\bar{v}_2,v_2)+R(v_1,\bar{v}_2,\bar{v}_1,v_2)].
\end{equation}

The vector $w:=b\mu_2\bar{v}_1-a\mu_1\bar{v}_2$ is orthogonal to both $v$ and $A_p v$ so that $1=\sec(v,w)$. Equivalently $$(a^2\mu_1^2+b^2\mu_2^2)=R(a v_1+b v_2, b\mu_2\bar{v}_1-a\mu_1\bar{v}_2, b\mu_2\bar{v}_1-a\mu_1\bar{v}_2,a v_1+b v_2).$$  Expanding the above, using Lemma \ref{c0}, and simplifying yields

\begin{equation}\label{3}
\Phi=a^2b^2(\mu_2^2\lambda_1+\mu_1^2\lambda_2)+a^2\mu_1^2(a^2-1)+b^2\mu_2^2(b^2-1).
\end{equation}

Substituting (\ref{3}) into (\ref{1}) and simplifying using $a^2+b^2=1$ yields the desired formula for $\lambda(v)$.
\end{proof}

\begin{corollary}\label{c11}
If $\mu_1<\mu_2$, then $\lambda_1<\lambda_2$.
\end{corollary}

\begin{proof}
In the notation of Lemma \ref{c1}, choose the vector $v$ so that $a=b=\sqrt{2}/2$.  As $\mu_1<\mu_2$, the vectors $v=a v_1+b v_2$ and $A_p^2 v=-(a \mu_1^2 v_1+b \mu_2^2 v_2)$ are linearly independent.  Therefore $\lambda(v)=\sec(v,A_pv)<\sec(A_pv,A_p^2v)=\lambda(A_pv/\|A_pv\|)$.  By Lemma \ref{c1}, $\frac{1}{2}\lambda_1+\frac{1}{2}\lambda_2< \frac{\mu_1^2}{\mu_1^2+\mu_2^2}\lambda_1+\frac{\mu_2^2}{\mu_1^2+\mu_2^2}\lambda_2$, or equivalently, $(\frac{1}{2}-\frac{\mu_1^2}{\mu_1^2+\mu_2^2})\lambda_1<(\frac{\mu_2^2}{\mu_1^2+\mu_2^2}-\frac{1}{2})\lambda_2$.  If $\lambda_2 \leq \lambda_1$, it follows that  $\frac{1}{2}-\frac{\mu_1^2}{\mu_1^2+\mu_2^2}<\frac{\mu_2^2}{\mu_1^2+\mu_2^2}-\frac{1}{2}$, a contradiction.
\end{proof}

Given unit vectors $e_i \in \sigma_i$, $i=1,2$, consider the following components of the curvature tensor: $\alpha=R(e_1,\bar{e}_1,e_{2},\bar{e}_2)$, $b=R(\bar{e}_1,e_2,e_1,\bar{e}_2)$, and $\gamma=R(e_2,e_1,\bar{e}_1,\bar{e}_2)$.  By the Bianchi identity,
\begin{equation}\label{bianchi}
\alpha+\beta+\gamma=0.
\end{equation}










\begin{lemma}\label{c2}
In the notation above, $\beta=\gamma>0$, $\alpha=-2\gamma<0$, and $\mu_1^2(\lambda_2-1)+\mu_2^2(\lambda_1-1)=6\mu_1\mu_2\gamma$.  Moreover, $(\lambda_1-1)(\lambda_2-1) \leq 9\gamma^2$.
\end{lemma}

\begin{proof}
Set $v_1=e_1$, $v_2=e_2$, and $a=b=\frac{\sqrt{2}}{2}$ and use (\ref{2}) and (\ref{3}) to deduce
\begin{equation}\label{4}
\mu_1^2(\lambda_2-1)+\mu_2^2(\lambda_1-1)=2\mu_1\mu_2(\beta-\alpha).
\end{equation}
after simplification.

Similarly, set $v_1=e_1$, $v_2=\bar{e}_2$, and $a=b=\frac{\sqrt{2}}{2}$ and use (\ref{2}) and (\ref{3}) to deduce
\begin{equation}\label{6}
\mu_1^2(\lambda_2-1)+\mu_2^2(\lambda_1-1)=2\mu_1\mu_2(\gamma-\alpha).
\end{equation}
As $\mu_i>0$ and $\lambda_i>1$, (\ref{4}) and (\ref{6}) imply that $\beta=\gamma$.  By (\ref{bianchi}) $\alpha=-2\gamma$ which upon substitution into (\ref{6}) yields $$\mu_1^2(\lambda_2-1)+\mu_2^2(\lambda_1-1)=6\mu_1\mu_2\gamma$$ from which the remaining inequalities are easily deduced.
\end{proof}

\begin{lemma}\label{c3}
In the notation above, $\lambda_1 \leq 3 \gamma+1\leq \lambda_2$. Equality holds in either case only if $\lambda_1=\lambda_2=3\gamma+1$ and $\mu_1=\mu_2.$
\end{lemma}

\begin{proof}
If $\lambda_1 \geq 3 \gamma+1$, then $9\gamma^2\leq (\lambda_1-1)^2\leq (\lambda_1-1)(\lambda_2-1)\leq 9\gamma^2$, implying that $\lambda_1=\lambda_2 =3 \gamma +1$  (and $\mu_1=\mu_2$ by Corollary \ref{c11}).  
Lemma \ref{c2} and the derivation of Berger's curvature inequality \cite{ber2,ka} imply
\begin{align}\label{long}
&2\gamma=-\alpha= R(\bar{e}_1,e_1,e_2,\bar{e}_2)=\nonumber\\
&\frac{1}{6}[\sec(\bar{e}_1+\bar{e}_2,e_1+e_2) +\sec(e_1+\bar{e}_2,\bar{e}_1-e_2)]\nonumber\\
&+\frac{1}{6}[\sec(\bar{e}_1-\bar{e}_2,e_1-e_2)+\sec(e_1-\bar{e}_2,\bar{e}_1+e_2)]\nonumber\\
&-\frac{1}{6}[\sec(\bar{e}_1-\bar{e}_2,e_1+e_2)+\sec(e_1-\bar{e}_2,\bar{e}_1-e_2)]\nonumber\\
&-\frac{1}{6}[\sec(\bar{e}_1+\bar{e}_2,e_1-e_2)+\sec(e_1+\bar{e}_2,\bar{e}_1+e_2)].\nonumber
\end{align}


If $\sigma\subset V=\sigma_1\oplus \sigma_2$ is a $2$-plane section and $v \in \sigma$ is a unit vector, then $\sec(\sigma) \leq \lambda(v) \leq \lambda_2$ where the last inequality is a consequence of Lemma \ref{c1}.  Hence $1 \leq \sec \leq \lambda_2$ on $V$.  These inequalities and the above formula for $2 \gamma$ yields the inequality $\lambda_2 \geq 3 \gamma+1$ where equality holds only if $\sec(\bar{e}_1+\bar{e}_2,e_1+e_2)=\lambda_2$.  Hence, equality holds only if $\lambda_2=\sec(\bar{e}_1+\bar{e}_2,e_1+e_2)\leq \lambda((e_1+e_2)/\sqrt{2})=\frac{1}{2}(\lambda_1+\lambda_2)\leq \lambda_2$, or equivalently if $\lambda_1=\lambda_2$ (and $\mu_1=\mu_2$ by Corollary \ref{c11}). 
\end{proof}

\begin{lemma}\label{k2}
For each nonzero vector $v \in T_pM$, there exists $c(v) \in \R \setminus \{0\}$ such that $A_p J_p v=c(v) J_p A_p v$.
\end{lemma}

\begin{proof}
Let $v \in T_pM \setminus \{0\}$.  Lemma \ref{crit} and (\ref{ghost}) imply that $J_p(E_v)=E_{J_pv}$. Therefore 

\begin{align}
&\spn\{J_pv, E_{J_pv}\}=\spn\{J_pv,J_p(E_v)\}=\nonumber\\
&J_p(\spn\{v,E_v\})=J_p((A_pv)^{\perp})=J_p(A_pv)^{\perp}\nonumber,
\end{align}

 where the last equality uses the fact that $J_p$ acts orthogonally. Conclude that both the vectors $J_p A_p v$ and $A_p J_pv $ are perpendicular to the codimension one subspace $\spn\{J_p v, E_{J_p v}\}$, concluding the proof.
 \end{proof}

\begin{lemma}\label{k3}
Either $A_pJ_p=J_pA_p$ or $A_pJ_p=-J_pA_p$.
\end{lemma}

 \begin{proof}
As both $A_pJ_p$ and $J_pA_p$ are non-degenerate, Lemma \ref{k2} implies that there is a nonzero constant $c\in \R$ such that $A_pJ_p=cJ_pA_p$.  Taking the determinant yields $c^{d}=1$, whence $c=\pm1$ since $d$ is even.   \end{proof}
 


\begin{proposition}\label{k5}
$A_pJ_p=J_pA_p$. 
\end{proposition}

\begin{proof}
Let $\sigma_1$ be an $A_p$-invariant $2$-plane section.  If $\sigma_1$ is $J_p$-invariant, then the restriction of $A_p$ and $J_p$ to $\sigma_1$ differ by a scalar, hence commute, concluding the proof in this case by Lemma \ref{k3}.

Hence, if the proposition fails, then $A_pJ_p=-J_pA_p$ and  $\sigma_1$ is not invariant under $J_p$.  The following derives a contradiction.

Let $\{e_1,e_2\}$ be an orthonormal basis of $\sigma_1$.  There exists a nonzero constant $\mu$ such that $A_p e_1=\mu e_2$ and $A_p e_2=-\mu e_1$.  Rescale $A_p$ and replace $e_2$ with $-e_2$, if necessary, so that $\mu=1$.  If $A_p^{*}$ denotes the adjoint of $A_p$, then $A_p^{*}=-A_p$ on the subspace $\sigma_1$.

  As $J_p$ is orthogonal, $\{e_3=J_p e_1,e_4=J_p e_2\}$ is an orthonormal basis of $\sigma_2:=J_p(\sigma_1)$.  The following calculations will demonstrate that $\{e_1,e_2,e_3,e_4\}$ form an orthonormal $4$-frame. As $\g_p(e_1,e_3)=\g_p(e_1,J_p e_1)=0=\g_p(e_2,J_p e_2)=\g_p(e_2,e_4)$, it remains to verify the equalities $\g_p(e_1,e_4)=0=\g_p(e_2,e_3).$  Calculate
 \begin{align}
&\g_p(e_1,e_4)=\g_p(e_1,J_p e_2)=\g_p(e_1,J_pA_p e_1)=\nonumber \\
&\g_p(e_1,-A_pJ_p e_1)=\g_p(-A_p^{*} e_1,J_p e_1)=\nonumber\\
&\g_p(A_p e_1,J_pe_1)=\g_p(e_2,J_p e_1)=\nonumber \\
&\g_p(J_p e_2,-e_1)=-\g_p(e_4,e_1)\nonumber
\end{align}
to conclude that $\g_p(e_1,e_4)=0$.  Finally, 
\begin{equation}
\g_p(e_2,e_3)=\g_p(J_p e_2,J_p e_3)=\g_p(e_4,-e_1)=0\nonumber
\end{equation} concluding the proof that $\{e_1,e_2,e_3,e_4\}$ are orthonormal. Let $\lambda=\sec(\sigma_1)$ and note that 
$$\lambda=R(e_1,e_2,e_2,e_1)=R(J_p e_1,J_p e_2,J_p e_2,J_p e_1)=R(e_3,e_4,e_4,e_3)=\sec(\sigma_2).$$  As $A_p(\sigma_2)=A_pJ_p(\sigma_1)=-J_pA_p(\sigma_1)=-J_p(\sigma_1)=\sigma_2$, the $2$-plane $\sigma_2$ is $A_p$-invariant.  Lemma \ref{c1} implies that $\lambda$ is the maximum sectional curvature on the subspace $V=\sigma_1\oplus \sigma_2$.  By Berger's curvature inequality (\cite{ber2,ka}), $R(e_1,e_2,e_4,e_3) \leq \frac{2}{3}(\lambda-1)$.  Consequently,
$$\lambda=R(e_1,e_2,e_2,e_1)=R(e_1,e_2,J_p e_2,J_p e_1)=R(e_1,e_2,e_4,e_3) \leq \frac{2}{3}(\lambda-1),$$ or equivalently, $\lambda \leq -2$, a contradiction.
\end{proof}

\begin{lemma}\label{k6}
Assume that $V=\sigma_1 \oplus \sigma_2$ is an orthogonal sum of $A_p$-invariant  and $J_p$-invariant $2$-plane sections.  Let $v_i \in \sigma_i$ be unit vectors and $\bar{v}_i=A_pv_i/\mu_i$.  If $J_p v_1=\bar{v}_1$, then $J_p v_2=\bar{v}_2$.  If $J_p v_1=-\bar{v}_1$, then $J_p v_2=-\bar{v}_2$.  In both cases $\gamma=R(v_2,v_1,\bar{v}_1,\bar{v}_2)=1$.
\end{lemma}

\begin{proof}
The assumptions imply that there are constants $c_1,c_2 \in \{-1,1\}$ such that $J_p v_i=c_i\bar{v}_i$ for $i=1,2$.  The first assertion in the lemma is the equality $c_1=c_2$ as will now be demonstrated.  Note that
$$\gamma=R(v_2,v_1,\bar{v}_1,\bar{v}_2)=R(J_pv_2,J_pv_1,\bar{v}_1,\bar{v}_2)=R(c_2 \bar{v}_2,c_1 \bar{v}_1,\bar{v}_1,\bar{v}_2)=c_1c_2$$ where Lemma \ref{c0} is used in the last equality.  By Lemma \ref{c2}, $\gamma>0$ whence $c_1=c_2$ and $\gamma=1$.
\end{proof}

\begin{corollary}\label{k7}
If $\sigma \subset T_pM$ is a $2$-plane section satisfying $A_p(\sigma)=\sigma$, then $J_p(\sigma)=\sigma$.
\end{corollary}

\begin{proof}
After possibly rescaling $A_p$, there exists an orthonormal basis $\{e_1,e_2\}$ of $\sigma$ satisfying $A_p e_1=e_2$ and $A_p e_2=-e_1$.  If $J_p(\sigma) \neq \sigma$ then $J_p(\sigma) \cap \sigma=\{0\}$.  Letting $e_3=J_p e_1$ and $e_4=J_p e_2$, the vectors $\{e_1,e_2,e_3,e_4\}$ span a $4$-dimensional subspace of $T_pM$.  

By Proposition \ref{k5}, $A_p e_3=e_4$ and $A_p e_4=-e_3$ since
$$A_p e_3=A_pJ_pe_1=J_pA_pe_1=J_p e_2=e_4$$ and $$A_p e_4=A_p J_pe_2=J_pA_p e_2=-J_p e_1=-e_3.$$
Let $v_1=\frac{e_1+e_4}{\sqrt{2}}$ and $v_2=\frac{e_1-e_4}{\sqrt{2}}$ and use the above to calculate $\bar{v}_1=\frac{e_2-e_3}{\sqrt{2}}$ and $\bar{v}_2=\frac{e_2+e_3}{\sqrt{2}}$.  Verify that $\sigma_1=\spn\{v_1,\bar{v}_1\}$ and $\sigma_2=\spn\{v_2,\bar{v}_2\}$ are orthogonal $A_p$-invariant and $J_p$-invariant $2$-planes and that $J_p v_1=-\bar{v}_1$ and $J_p v_2=\bar{v}_2$. This contradicts Lemma \ref{k6}.
\end{proof}

\subsection{Proof of Theorem \ref{thm:D} when $d=\dim_{\mathbb{R}}(M) \geq 6$}

\begin{lemma}\label{k8}
For $p \in M$, $A_p$ has at most two distinct eigenvalues.
\end{lemma}

\begin{proof}
If not, then there exist three orthogonal $A_p$-invariant $2$-planes $\sigma_i$, $i=1,2,3$ and constants $0<\mu_1<\mu_2<\mu_3$ such that $\|A(w_i)\|=\mu_i$ for each unit vector $w_i\in \sigma_i$.  Let $\lambda_i=\sec(\sigma_i)$.  As $\mu_1<\mu_2$, Corollary \ref{c11} implies that $\lambda_1<\lambda_2$.  By Lemmas \ref{c3} and \ref{k6},  $\lambda_2>4$.  As $\mu_2<\mu_3$, Corollary \ref{c11} implies that $\lambda_2<\lambda_3$.  By Lemmas \ref{c3} and \ref{k6},  $\lambda_2<4$, a contradiction.
\end{proof}

\begin{lemma}\label{k9}
If $d=\dim_{\mathbb{R}}(M)\geq 6$, then $A_p$ has a single eigenvalue for each $p \in M$.
\end{lemma}

\begin{proof}
If not, Lemma \ref{k8} implies that there exist constants $0<\mu_1<\mu_2$ and $A_p$-eigenspaces $E_1$ and $E_2$ such that $T_pM$ is the orthogonal direct sum $T_pM=E_1\oplus E_2$ and $\|A_p(v_i)\|=\mu_i$ for each unit vector $v_i \in E_i$, $i=1,2$.  As $\dim_{\mathbb{R}}(M)\geq 6$, one of the two eigenspaces $E_1$ or $E_2$ has real dimension at least four.\\

\textbf{Case I:} $\dim_{\mathbb{R}}(E_1)\geq 4$\\

Choose orthogonal $A_p$-invariant $2$-planes $\sigma_1, \sigma_2 \subset E_1$ and $\sigma_3 \subset E_2$.  Let $\lambda_i=\sec(\sigma_i)$ for each $i=1,2,3$.    As $\mu_1<\mu_2$, Corollary \ref{c11} implies that $\lambda_1<\lambda_3$ and $\lambda_2<\lambda_3$.  Apply Lemmas \ref{c3} and \ref{k6}  to the four dimensional subspaces $\sigma_1\oplus \sigma_3$ and $\sigma_2 \oplus \sigma_3$ to deduce $\lambda_1<4$ and $\lambda_2<4$. Apply Lemmas \ref{c2} and \ref{k6} to the four dimensional subspace $\sigma_1 \oplus \sigma_2$ to deduce $\lambda_1+\lambda_2=8$, a contradiction.\\

\textbf{Case II:} $\dim_{\mathbb{R}}(E_2)\geq 4$\\

Choose orthogonal $A_p$-invariant $2$-planes $\sigma_1 \subset E_1$ and $\sigma_2,\sigma_3 \subset E_2$.  Let $\lambda_i=\sec(\sigma_i)$ for each $i=1,2,3$.  As $\mu_1<\mu_2$, Corollary \ref{c11} implies that $\lambda_1<\lambda_2$ and $\lambda_1<\lambda_3$.  Apply Lemmas \ref{c3} and \ref{k6}  to the four dimensional subspaces $\sigma_1\oplus \sigma_2$ and $\sigma_1 \oplus \sigma_3$ to deduce $\lambda_2>4$ and $\lambda_3>4$.  Applying Lemmas \ref{c2} and \ref{k6} to the four dimensional subspace $\sigma_2 \oplus \sigma_3$ to deduce $\lambda_2+\lambda_3=8$, a contradiction.
\end{proof}

\begin{remark}\label{approach}
When $\dim_{\mathbb{R}}(M)\geq 6$, Theorem \ref{thm:D} is easily derived from Lemma \ref{k9} and Theorem \ref{thm:C}.  This approach is taken when $\dim_{\mathbb{R}}=4$ in the next section.  

In the remainder of this section, a more elementary proof is presented for the case when $\dim_{\mathbb{R}}(M)\geq 6$.  This alternative proof is based on the well-known classification \cite{haw, igu} of simply-connected K\"ahlerian manifolds having constant holomorphic curvatures.
\end{remark}

\begin{corollary}\label{k13}
A $2$-plane $\sigma\subset T_pM$ is holomorphic if and only if $A_p(\sigma)=\sigma$.
\end{corollary}

\begin{proof}
Fix $p \in M$ and let $\sigma \subset T_pM$ be a $2$-plane.  If $A_p(\sigma)=\sigma$ then $J_p(\sigma)=\sigma$ by Corollary \ref{k7}.  Conversely, assume that $J_p(\sigma)=\sigma$ and let $v \in \sigma$ be a nonzero vector.  The $2$-plane $\bar{\sigma}=\spn\{v,A_pv\}$ is $A_p$-invariant by Lemma \ref{k9}.  By Corollary \ref{k7}, $\bar{\sigma}$ is $J_p$-invariant.  As $v$ lies in a unique holomorphic $2$-plane, $\sigma=\bar{\sigma}$, so that $\sigma$ is $A_p$-invariant. 
\end{proof}

\begin{corollary}\label{k10}
If $d=\dim_{\mathbb{R}}(M) \geq 6$, then $\lambda(v)=4$ for every unit vector $v \in SM$.
\end{corollary}

\begin{proof}
Given $v \in S_pM$, the $2$-plane $\sigma_1=\spn\{v, A_p v\}$ is $A_p$-invariant by Lemma \ref{k9}.  As $\dim_{\mathbb{R}}(M)\geq 6$, there exist orthogonal $A_p$-invariant $2$-planes $\sigma_2, \sigma_3 \subset \sigma_1^{\perp}$.  Let $\lambda_i=\sec(\sigma_i)$ for $i=1,2,3$ and note that $\lambda(v)=\lambda_1$.

Applying Lemmas \ref{c2} and \ref{k6} to the three four dimensional subspaces $\sigma_i \oplus \sigma_j$, $i,j \in \{1,2,3\}$ distinct, yields the linear system $$\lambda_1+\lambda_2=\lambda_1+\lambda_3=\lambda_2+\lambda_3=8,$$  whose solution $\lambda_1=\lambda_2=\lambda_3=4$ is unique.
\end{proof}

\begin{theorem}\label{not4}
A K\"ahlerian manifold with $\sec \geq 1$, real dimension $d \geq 6$, and spherical rank at least $d-2$ is isometric to a globally symmetric $\mathbb{C}\mathbb{P}^{d/2}$ with holomorphic curvatures equal to $4$.
\end{theorem}

\begin{proof}
It suffices to prove that all holomorphic $2$-planes in $M$ have sectional curvature equal to four by \cite{haw,igu}.  Let $p \in M$ and  let $\sigma\subset T_pM$ be a holomorphic $2$-plane.  Let $v \in \sigma$ be a nonzero vector.  By Corollary \ref{k13}, $\sigma$ is $A_p$-invariant, so that $\sec(\sigma)=\sec(v,A_pv)=\lambda(v)$.  By Corollary \ref{k10}, $\lambda(v)=4$. 
\end{proof}

\section{Proof of Theorem \ref{thm:D} in real dimension four}\label{sec:D2}
This final section completes the proof of Theorem \ref{thm:D}, establishing its veracity when $d=\dim_{\mathbb{R}}(M)=4$.  The approach, alluded to in Remark \ref{approach}, is to appeal to Theorem \ref{thm:C}.  The main step in proving that $M$ satisfies the Raki\'c duality principle is to establish the analogue of Lemma \ref{k9} when $d=4$.  The following lemma, likely well-known, is used for this purpose.

\begin{lemma}\label{split}
Let $B$ be an open connected subset of a Riemannian manifold $(M,\g)$ admitting a pair of orthogonal and totally geodesic foliations $\mathcal{F}_1$ and $\mathcal{F}_2$.  Then $B$ is locally isometric to the product $\mathcal{F}_1 \times \mathcal{F}_2$.
\end{lemma}

\begin{proof}
If $H=T\mathcal{F}_1$ and $V=T\mathcal{F}_2$, then the tangent bundle splits orthogonally $TB=H\oplus V$. By de Rham's splitting theorem, it suffices to prove that the distribution $H$ is parallel on $B$. Let $h, \bar{h}$ denote vector fields tangent to $H$ and let $v, \bar{v}$ denote vector fields tangent to $V$.  

As $H$ is integrable, $0=\g([h,\bar{h}],v),$ implying $\g(\nabla_h \bar{h},v)=\g(\nabla_{\bar{h}} h,v)$.  As $H$ is totally geodesic, $\g(\nabla_h \bar{h},v)=-\g(\nabla_{\bar{h}} h,v)$.  Conclude that  
\begin{equation}\label{y1}
\g(\nabla_{\bar{h}} h,v)=0.
\end{equation}

Similarly, the fact that $V$ is integrable and totally geodesic implies that $\g(\nabla_{\bar{v}} v,h)=0$.  As $H$ and $V$ are orthogonal, this implies
\begin{equation}\label{y2}
\g(\nabla_{\bar{v}} h, v)=0.
\end{equation}

By (\ref{y1}) and (\ref{y2}), $H$ is parallel on $B$, concluding the proof.
\end{proof}

Recall from the proof of item (3) of Theorem \ref{thm:B} that there exists a smooth section $p \ni M \mapsto A_p \in SL(T_pM)$.




\begin{proposition}\label{k12}
Assume that $d=\dim_{\mathbb{R}}(M)=4$.  Then for each $p \in M$, $A_p$ has a single eigenvalue.
\end{proposition}

\begin{proof}

If not, then there exists a metric ball $B$ in $M$ with the property that for each $b \in B$, $A_b$ has two distinct eigenvalues.  For each $b \in B$, there exist constants $0<\mu_1(b)<\mu_2(b)$ and orthogonal eigenplanes $\sigma_1(b)$ and $\sigma_2(b)$ of $A_b$ satisfying $\|A_b(v_i)\|=\mu_i(b)$ for each unit vector $v_i \in \sigma_i(b)$.  As the $A_b$ vary smoothly with $b \in B$, the functions $\mu_i:B \rightarrow \R$ and the orthogonal splitting $TB=\sigma_1 \oplus \sigma_2$ are both smooth.  Define $\lambda_i:B \rightarrow \R$ by $\lambda_i(b)=\sec(\sigma_i(b))$ for $i=1,2$.

After possibly reducing the radius of $B$,  there exist smooth unit vector fields $v_1$ and $v_2$ on $B$ tangent to $\sigma_1$ and $\sigma_2$ respectively.  By Corollary \ref{k7}, the two $2$-plane fields $\sigma_1$ and $\sigma_2$ are $J$-invariant.  Therefore, letting $\bar{v}_i=Jv_i$, the smooth orthonormal framing $\{v_1,\bar{v}_1,v_2,\bar{v}_2\}$ of $TB$ satisfies $\sigma_i=\spn\{v_i,\bar{v}_i\}$ for $i=1,2$.  Define $\gamma: B \rightarrow \mathbb{R}$ by $\gamma=R(v_2,v_1,\bar{v}_1,\bar{v}_2)$.   Again by Corollary \ref{k7}, the $A_b$-invariant $2$-planes $\sigma_i(b)$ are $J_b$-invariant and by Lemma \ref{k6}, $\gamma=1$ on $B$.  

Corollary \ref{c11}, implies that $\lambda_1(b)<\lambda_2(b)$ and Lemma \ref{c3} implies 
 \begin{equation}\label{ineq}
 \lambda_1(b)<4<\lambda_2(b)
\end{equation} for each $b\in B$.

The goal of the following calculations is to show that the orthogonal distributions $\sigma_1$ and $\sigma_2$ are integrable and totally geodesic.  As $J$ is parallel, 
\begin{equation}\label{par}
\g(\nabla_X JY,Z)=\g(J \nabla_X Y,Z)=-\g(\nabla_X Y,JZ)
\end{equation} for all smooth vector fields $X,Y,Z$.

Use (\ref{par}) to conclude 
\begin{equation}\label{t1}
\g(\nabla_{v_2}v_2, \bar{v}_1)=-\g(\nabla_{v_2} \bar{v}_2 ,v_1).
\end{equation}

Use the differential Bianchi identity, $$0=(\nabla_{v_2}R)(v_1,\bar{v}_1,v_1,v_2)+(\nabla_{v_1}R)(\bar{v}_1,v_2,v_1,v_2)+(\nabla_{\bar{v}_1}R)(v_2,v_1,v_1,v_2)$$ to derive

\begin{equation}\label{b1}
(\lambda_1-1)\g(\nabla_{v_2}v_2,\bar{v}_1)+3\g(\nabla_{v_2}\bar{v}_2,v_1)=0.
\end{equation}

Use (\ref{ineq}), (\ref{t1}), and (\ref{b1}) to conclude

\begin{equation}\label{b5}
\g(\nabla_{v_2}v_2, \bar{v}_1)=\g(\nabla_{v_2} \bar{v}_2 ,v_1)=0.
\end{equation}

Set $w_1:=\bar{v}_1$ and $\bar{w}_1:=J w_1=-v_1$.  Repeating the above calculations with $w_1$ and $\bar{w}_1$ in place of $v_1$ and $\bar{v}_1$, respectively, yields the following analogue of (\ref{b5})

\begin{equation}
\g(\nabla_{v_2}v_2, \bar{w}_1)=\g(\nabla_{v_2} \bar{v}_2 ,w_1)=0,
\end{equation}
or equivalently,

\begin{equation}\label{b6}
\g(\nabla_{v_2} v_2,v_1)=\g(\nabla_{v_2} \bar{v}_2,\bar{v}_1)=0.
\end{equation}

Set $w_2:=\bar{v}_2$ and $\bar{w}_2:=J w_2=-v_2$.  Repeating the above calculations with $w_2$ and $\bar{w}_2$ in place of $v_2$ and $\bar{v}_2$, respectively, yields the following analogues of (\ref{b5}) and (\ref{b6})

\begin{equation}
\g(\nabla_{w_2}w_2, \bar{v}_1)=\g(\nabla_{w_2} \bar{w}_2 ,v_1)=0,
\end{equation}

and

\begin{equation}
\g(\nabla_{w_2}w_2, v_1)=\g(\nabla_{w_2} \bar{w}_2 ,\bar{v}_1)=0,
\end{equation}
or equivalently,

\begin{equation}\label{b11}
\g(\nabla_{\bar{v}_2} \bar{v}_2,\bar{v}_1)=\g(\nabla_{\bar{v}_2} v_2,v_1)=0,
\end{equation}
and 

\begin{equation}\label{b12}
\g(\nabla_{\bar{v}_2} \bar{v}_2,v_1)=\g(\nabla_{\bar{v}_2}v_2,\bar{v}_1)=0.
\end{equation}

The $2$-plane field $\sigma_2$ is integrable and totally geodesic by (\ref{b5}), (\ref{b6}), (\ref{b11}), and (\ref{b12}).

Switching the roles of the indices $1$ and $2$ in the differential Bianchi calculation above, yields the following analogue of (\ref{b1})

\begin{equation}\label{b90}
(\lambda_2-1)\g(\nabla_{v_1} v_1,\bar{v}_2)+3\g(\nabla_{v_1} \bar{v}_1,v_2)=0.
\end{equation}

Now, arguing as in the case of the $2$-plane field $\sigma_2$, the $2$-plane field $\sigma_1$ is also integrable and totally geodesic.  As the tangent $2$-plane fields $\sigma_1$ and $\sigma_2$ are orthogonal, integrable, and totally geodesic, $B$ is locally isometric to a Riemannian product by Lemma \ref{split}.  This contradicts the curvature assumption $\sec\geq 1$.
\end{proof}

\begin{theorem}\label{4}
A K\"ahlerian manifold with $\sec \geq 1$, real dimension $d=4$, and spherical rank at least $2$ is isometric to a globally symmetric $\mathbb{C}\mathbb{P}^{2}$ with holomorphic curvatures equal to $4$.
\end{theorem}

\begin{proof}
It suffices to prove that $M$ satisfies the Raki\'c duality principle by Theorem \ref{thm:C}.  

Let $p \in M$ and let $v,w \in S_pM$ be a pair of orthonormal vectors.  The Jacobi operator $\mathcal{J}_v$ has two eigenspaces, namely the two-dimensional $1$-eigenspace $E_v$ and the one-dimensional $\lambda(v)$-eigenspace spanned by the vector $A_pv$.  Similarly, the Jacobi operator $\mathcal{J}_w$ has a two-dimensional $1$-eigenspace $E_w$ and a one-dimensional $\lambda(w)$-eigenspace spanned by $A_pw$.

If $w \in E_v$, then $v \in E_w$ by Lemma \ref{crit}.  If $w$ lies in the $\lambda(v)$-eigenspace of $\mathcal{J}_v$, then $w$ is a multiple of $A_pv$.  By Proposition \ref{k12} the $2$-plane $\sigma:=\spn\{v,w\}$ is $A_p$-invariant, whence $\lambda(w)=\sec(\sigma)=\lambda(v)$ and $v$ lies in the $\lambda(w)$-eigenspace of $\mathcal{J}_w$.
\end{proof}

Together, Theorems \ref{not4} and \ref{4} complete the proof of Theorem \ref{thm:D}.

\end{document}